\documentclass[12pt]{amsart}
\pdfoutput=1
\usepackage{microtype}
\usepackage{etex} 
\usepackage[active]{srcltx}
\usepackage{calc,amssymb,amsthm,amsmath,amscd, eucal,ulem, multirow}
\usepackage{alltt}
\RequirePackage[dvipsnames,usenames]{color}

\input{mabliautoref.sty}
\input{xy}
\xyoption{all}
\input{kmacros3.sty}
\usepackage{tikz}
\usepackage{amsfonts, mathrsfs}
\usepackage{mathtools}
\usepackage{calligra}
\usepackage{stmaryrd} 
\usepackage{dcpic, pictexwd}
\usepackage[left=1in,top=1in,right=1in,bottom=1in]{geometry}
\usepackage{bm}
\usepackage{verbatim}
\usepackage{accents}

\usepackage[cal=boondox, calscaled=1.05]{mathalfa}

\numberwithin{equation}{theorem}

\newcommand{\kay}{\mathcal{k}}
\newcommand{\el}{\mathcal{l}}

\renewcommand{\:}{\colon}

\DeclareMathOperator{\vol}{vol}

\DeclareMathOperator{\Bl}{Bl}
\DeclareMathOperator{\End}{End}
\DeclareMathOperator{\Cox}{Cox}

\DeclareMathOperator{\DIV}{Div}
\DeclareMathOperator{\Cl}{Cl}
\usepackage{tikz-cd}

\usepackage{caption}

\DeclareMathOperator{\eff}{\overline{Eff}} 
\DeclareMathOperator{\mov}{\overline{Mov}}
\DeclareMathOperator{\movd}{\overline{Mov}^1}
\DeclareMathOperator{\nef}{Nef}
\DeclareMathOperator{\ample}{Amp}
\DeclareMathOperator{\bigc}{Big}

\DeclareMathOperator{\FSupp}{FS}
\DeclareMathOperator{\BFS}{BFS}
\DeclareMathOperator{\AFS}{AFS}
\DeclareMathOperator{\NFS}{NFS}
\DeclareMathOperator{\Frob}{Frob}
\DeclareMathOperator{\FE}{\overline{FE}}

\DeclareMathOperator{\moriCone}{\overline{NE}}

\usepackage{setspace}
\usepackage{hyperref}
\usepackage{url}
\let\oldFootnote\footnote
\newcommand\nextToken\relax

\renewcommand\footnote[1]{%
    \oldFootnote{#1}\futurelet\nextToken\isFootnote}

\newcommand\isFootnote{%
    \ifx\footnote\nextToken\textsuperscript{,}\fi}

\usepackage{enumitem}
\usepackage{graphicx}
\usepackage[all,cmtip]{xy}

\usepackage{verbatim}

\theoremstyle{theorem}

\usepackage{marvosym}

\begin{document}
\title{The geometry of Frobenius on toric varieties} 
\author[J.~Carvajal-Rojas]{Javier Carvajal-Rojas}
\address{Centro de Investigaci\'on en Matem\'aticas, A.C., Callej\'on Jalisco s/n, 36024 Col. Valenciana, Guanajuato, Gto, M\'exico}
\email{\href{mailto:javier.carvajal@cimat.mx}{javier.carvajal@cimat.mx}}
\author[E. A.~\"Ozavcı]{Emre Alp \"Ozavcı}
\address{\'Ecole Polytechnique F\'ed\'erale de Lausanne\\ SB MATH CAG\\MA C3 635 (B\^atiment MA)\\ Station 8 \\CH-1015 Lausanne\\Switzerland}
\email{\href{mailto:emre.ozavci@epfl.ch}{emre.ozavci@epfl.ch}}

\thanks{Carvajal-Rojas was partially supported by the grants ERC-STG \#804334, FWO \#G079218N, CONAHCYT \#CBF2023-2024-224 and \#CF-2023-G-33. Özavcı was partially supported by ERC-STG \#804334.}


\begin{abstract}
    We give a geometric description of the positivity of the Frobenius-trace kernel on a $\mathbb{Q}$-factorial projective toric variety. To do so, we define its Frobenius support as well as the notions of $F$-effectiveness for divisors and $1$-cycles. As it turns out, the interaction of the corresponding cone of $F$-effective curves with the Mori cone of curves reflects the type of extremal Mori contractions that the variety can undergo. As a corollary, we obtain that the Frobenius-trace kernel is ample if and only if the Picard rank is $1$.
\end{abstract}

\maketitle


\section{Introduction} \label{sec.Intro}

Ever since differential calculus was discovered, geometric folk wisdom states that the geometry of a variety lives in its tangent space (or (co-)tangent sheaf in more algebro-geometric terms). That is, the basic geometry of a variety is supposed to be read from its tangent space. For instance, we can tell whether or not a variety is smooth via the jacobian criterion. In more global terms, we have Mori's striking result \cite{MoriCharacterizationProjSpace} (originally known as Hartshorne's conjecture \cite{HartshorneAmpleSubvarieties}), asserting that a projective variety is a projective space if and only if its tangent sheaf is ample. This work served as the basis for the modern Minimal Model Program (see \cite{KollarMori}) aimed at classifying varieties---up to proper birationality---via the positivity of the canonical sheaf, i.e., the determinant of the cotangent sheaf.

However, in characteristic $p>0$, the differential approach to algebraic geometry is not as smooth as in characteristic zero. The first bump we face is that we may take a local coordinate $x$ of our variety and then get a new one $y
\coloneqq x^p$ such that $dy = px^{p-1} dx = 0$ without $y$ being constant. This leads to all sorts of issues that are often collectively referred to as positive characteristic pathologies. 

A typical and familiar example of these pathologies originates in the study of elliptic curves and, specifically, their endomorphism ring. It is known that for an elliptic curve $E$ over an algebraically closed field of characteristic zero $\End_{\bZ}(E)$ is an abelian group of rank $1$ or $2$, in the former case $E$ is said to have \emph{complex multiplication}. However, if the base field has positive characteristic, a third option may occur in which $\End_{\bZ}(E)$ has rank $4$ and $E$ is referred to as
\emph{super-singular} if so. If $\End_{\bZ}(E)$ has rank $1$ or $2$, $E$ is said to be \emph{ordinary}. In general, ordinary elliptic curves behave as those in characteristic zero and super-singular ones exhibit a wild, new behavior. 

To better understand these pathologies, it has been productive to observe that the mapping $x \mapsto x^p$ on local coordinates is actually a ring homomorphism as $p$ divides $\binom{p}{i}$ for all $i=1,
\ldots, p-1$. This means that there is an induced endomorphism $F \: X \to X$ on the variety $X$ known as the \emph{Frobenius map} of $X$. For example, if $X$ were an elliptic curve, it would be ordinary if and only if it is \emph{$F$-split}---meaning that the map $F^{\#}\:\sO_X \to F_* \sO_X$ splits in the category of $\sO_X$-modules. Thus, whether or not an elliptic curve is ordinary is encoded in its Frobenius map.

The study of singularities and local algebra is another field in which the Frobenius map has been instrumental in dealing with anomalies of positive characteristic. Indeed, a celebrated theorem of Kunz \cite{KunzCharacterizationsOfRegularLocalRings} establishes that a variety $X$ is smooth if and only if its Frobenius map $F$ is flat. Thus, one may replace the jacobian criterion (and resolution of singularities) for Kunz's theorem in our approach to understanding singularities. This approach goes by the name of \emph{$F$-singularities} and it encompasses Hockster--Huneke's \emph{tight closure theory}. In this theory, one is supposed to read the properties of a singularity from its Frobenius map rather than its cotangent module. The success of this viewpoint has been enormous and it is hard to overstate how impactful it has all been and continues to be. For example, today it serves as the blueprint for the new foundations of singularities in mixed characteristics.

In view of this, it is natural to wonder how much of the global geometry of a projective variety is encoded in its Frobenius map. A first step towards this is searching for a projective analog of Kunz's theorem, or, say, a Frobenius-theoretic analog of Mori's theorem. This is the content of the work by Patakfalvi and the first named author \cite{CarvajalRojasPatakfalviVarietieswithAmpleFrobenius-traceKernel}. In there, they postulated that, in analogy to the tangent sheaf, the positivity of the \emph{Frobenius-trace kernel} should reflect the geometry of a variety. Their main result is that, in dimension $\leq 3$, the Frobenius-trace kernel is ample only for Fano varieties of Picard rank $1$.

In this work, our aim is to provide a geometric characterization of the positivity properties of the Frobenius-trace kernel in the simplest case of \emph{toric varieties}. That is, for the varieties that can be written down using binomial equations only (it turns out that every variety can be expressed using trinomial equations). We will reveal that indeed much of the birational geometry of a toric variety is governed by its Frobenius map.

We now explain some of our main results in the smooth case. Although we work out the general $\bQ$-factorial case, due to its more technical nature, we leave it to the more expert reader to study in the main text. See \autoref{thm.MainTheorem}.

Let $X$ be a $d$-dimensional smooth projective variety over an algebraically closed field $\kay$ of characteristic $p>0$. We may then consider its $e$-th Frobenius map $F^e\: X \to X$. Recall that $F^e$ is the identity on $X$ as a topological map but raises local sections to the $q \coloneqq p^e$-th power.\footnote{Since $\kay$ is perfect, we may identify $F^e$ with the relative Frobenius morphism of $X/\kay$.}

In this work, we investigate the negativity of the cokernel of $F^{e,\#}$
\[
0 \to \sO_X \xrightarrow{F^{e,\#}\: s \mapsto s^q}F_*^e \sO_X \to \sB_{X,e}^1 \to 0
\]
Equivalently, we examine the positivity of its dual, the \emph{kernel of the Frobenius trace}
\[
0 \to \sE_{X,e} \to F_*^e \omega_X^{1-q} \xrightarrow{\tau_X^e=(F^{e,\#})^{\vee}} \sO_X \to 0
\]
where $\omega_X \coloneqq \det \Omega^1_{X}$ is the canonical sheaf of $X$. By Kunz's theorem \cite{KunzCharacterizationsOfRegularLocalRings}, $\sB_{X,e}^1$ and so $\sE_{X,e}$ are locally free sheaves of rank $q^{d}-1$ as $X$ is smooth. Twisting by $\omega_X$ and using the projection formula, we obtain the following (perhaps better known) exact sequence
\[
0 \to \sE_{X,e} \otimes \omega_X \to F_*^e \omega_X \xrightarrow{\kappa_X^e = \tau_X^e \otimes \omega_X} \omega_X \to 0
\]
where $\kappa_X^e$ is the ($e$-th power of the) \emph{Cartier operator} on $X$ \cite{CartierUneNouvelleOperation}. In other words, $\kappa_X^e$ is the $\omega_X$-dual (also called the Serre dual) of $F^{e,\#}$. For details, see \cite[\S1.3]{BrionKumarFrobeniusSplitting}.

Assume now that $X$ is toric. According to Thomsen and Achinger \cite{ThomsenFrobeniusDirectImagesOfLineBundlesToricVarieties,AchingerCharacterizationOfToricVarieties}, $X$ is $F$-split (i.e., $F^{e,\#}$ slits as a map of $\sO_X$-modules) and $\sE_{X,e}$ splits as a direct sum
\[
\sE_{X,e} \simeq \bigoplus_{0 \neq [E] \in \Cl(X)} \sO_X(E)^{\oplus m(E;q)}
\]
where $m(E;q)$ is the number of \emph{torus-invariant divisors} on $X$ with coefficients in $\{0,\ldots,q-1\}$ being linearly equivalent to $q E$. In particular, $E$ is pseudo-effective if $m(E;q)\neq 0$. Hence $\sE_{X,e}$ is pseudo-effective. The leading question of this article is the following.
\begin{question} \label{que.MainQuestion}
    For which projective toric varieties $X$ is the Frobenius-trace kernel $\sE_{X,e}$ ample, big, and nef; respectively, for all $e>0$? 
\end{question}

\begin{remark}
The sheaves $\sE_{X,e}$ naturally form a surjective projective system
\[ \sE_{X,1} \twoheadleftarrow  \sE_{X,2} \twoheadleftarrow  \sE_{X,3} \twoheadleftarrow \cdots .\]
For details, see \cite[Remark 5.2]{CarvajalRojasPatakfalviVarietieswithAmpleFrobenius-traceKernel}. Therefore, in discussing the aforementioned positivity properties on $\sE_{X,e}$, it is the same to do so for all $e>0$ and all $e \gg 0$. Thus, we may and will interchange these two quantifiers freely in what follows.
\end{remark}



Our first main result starts by answering \autoref{que.MainQuestion} as follows.

\begin{theoremA*}[{\autoref{thm.MainTheoremBignessAmplenessProjSpace}}] 
With notation as above,  
the following statements are equivalent:
\begin{enumerate}
    \item $X \simeq \bP^d$.
    \item $\sE_{X,e}$ is ample for all $e>0$.
    \item $\sE_{X,e}$ is big for all $e>0$.
\end{enumerate}
\end{theoremA*}

One readily sees by direct computation that $\sE_{X,e}$ is ample for all $e>0$ if $X=\bP^d$; see \cite[Corollary 3.5]{CarvajalRojasPatakfalviVarietieswithAmpleFrobenius-traceKernel}. In addition, ampleness implies bigness, in general. Thus, the actual content of Theorem A is the implication (c)$\Longrightarrow$(a). We prove this in \autoref{sec.BIGNESS}; see \autoref{thm.MainTheoremBignessAmplenessProjSpace}

Our next task is then to determine for which $X$ the sheaves $\sE_{X,e}$ are nef for all $e>0$. Since for toric varieties global generation and nefness are the same notion, one readily concludes from \cite[Proposition 5.7]{CarvajalRojasPatakfalviVarietieswithAmpleFrobenius-traceKernel} that toric varieties for which $\sE_{X,e}$ is nef for some $e \in \bN$ are necessarily Fano. The question then becomes the following.
\begin{question}
Which Fano toric varieties have nef Frobenius-trace kernels?
\end{question}
 To answer this question, recall that the \emph{length} of an extremal ray $R$ of the Mori cone of $X$ of a smooth Fano variety $X$ is
\[
\ell(R) \coloneqq \min\{-K_X \cdot C \mid C \text{ is a rational curve such that } [C] \in R \}.
\]
Since $X$ is Fano, $\ell(R)>0$ for all extremal rays $R$ in its Mori cone. Loosely speaking, the longer the extremal rays of $X$ are, the more positively curved $X$ is and so it is closer to being a projective space. 

Given an extremal ray $R$, let $d_R$ be the maximal dimension of an exceptional fiber of the corresponding extremal contraction $\phi_R \:X \to X'$.\footnote{Where the dimension of a closed subspace is defined as the maximal dimension of an irreducible component.} Over $\bC$, the following bounds hold
\[
\ell(R) \leq \begin{cases}
    d_R+1  &\text{ if $\phi_R$  is a Mori fibration (\ie not birational)}, \\
    d_R & \text{ if $\phi_R$ is birational.}
\end{cases} 
\]
This is a consequence of the Ionescu--Wiśniewski inequalities \cite[Theorem 1.1]{WisniewskiInequality}. Moreover, $\ell(R) \geq d_R+1$ (if and) only if $\phi_R$ is a projective bundle \cite[Theorem 1.3]{HoringNovelliMoriContractionsOfMaximalLength}. Likewise, when $\phi_R$ is birational, then $\ell(R) \geq d_R$ is equivalent to $\phi_R$ being a smooth blowup, i.e., the blowup of a smooth variety along a smooth subvariety of codimension $\ell(R)+1$. See \cite[Theorem 5.1]{AndreattaOcchettaSpecialRaysInTheMoriConeOfaProjVar}. As mentioned above, these results were shown in characteristic zero. The authors are unaware of any progress towards them in positive characteristics. However, it can be seen that they hold for toric varieties in arbitrary characteristics; see \autoref{rem.BoundsForLengths}.

With the above in mind, we shall say that an extremal ray $R$ \emph{has maximal length} if either $\phi_R$ is a fibration and $\ell(R)\geq d_R + 1$ or $\phi_R$ is birational and $\ell(R) \geq d_R$. At least in characteristic zero or for toric varieties, this is to say that $\phi_R$ is either a projective bundle or a smooth blowup. We shall say that a Fano variety has \emph{maximal lengths} if all its extremal rays have maximal length. We will refer to such varieties as \emph{extremal Fano varieties}. In particular, a Fano variety has maximal lengths if all its extremal contractions are either projective bundles or smooth blowups, and the converse holds in characteristic zero or for toric varieties.

\begin{theoremB*}[{\autoref{thmNefnessOfE}, \autoref{cor.NefnessChainSmoothBlowyups}}]
With notation as above, $\sE_{X,e}$ is nef for all $e>0$ if and only if it is an extremal Fano variety. Moreover, in this case, there is a finite chain of smooth blowups between extremal Fano varieties
\[
X \to X_1 \to X_2 \to X_3 \to \cdots \to X_n
\]
ending in a homogeneous space $X_n$.
\end{theoremB*}

For $X$ as above, recall that the following statements are equivalent (see \cite{FujinoSatoToricVarietiesWhoseNefBundlesAreBig, ArzhantsevGaiffulinHomogeneousToricVarieties}):
\begin{enumerate}
    \item $X$ admits no birational extremal contractions (i.e., $\nef(X)=\eff(X)$).
    \item $X$ is a product of projective spaces.
    \item $X$ is a homogeneous space (i.e., it admits a transitive action of an algebraic group).
    \item $X$ has a nef tangent sheaf.
    \item $X$ has a globally generated tangent sheaf.
\end{enumerate}

This raises the question of what strengthening of nefness would characterize homogeneity and so any of these properties. For example, if $X$ is a (toric) del Pezzo surface then $\sE_{X,e}$ is nef for all $e>0$. There are, up to isomorphism, five toric del Pezzo surfaces, namely
\[
\bP^2, \, \bP^1 \times \bP^1, \, \Bl_{x_0}\bP^2, \, \Bl_{x_0,x_1} \bP^2, \, \Bl_{x_0,x_1,x_2} \bP^2,
\]
where $x_0\coloneqq [0:0:1]$, $x_1 \coloneqq [0:1:0]$, and $x_2 \coloneqq [1:0:0]$ are the torus-invariant points of $\bP^2$. Among these toric surfaces, only the first two are homogeneous spaces. How can we use the positivity of their Frobenius-trace kernels to tell them apart? One way is by means of a new invariant that we call \emph{ample $F$-signature}. We will define this invariant in \autoref{sec.AmpleFSignHomogeneity} but, for toric varieties, it can be computed as follows. 

Given $[E] \in \Cl(X)$, there is $\alpha(E) \in [0,1] \cap \bQ$ such that
\[
m(E;q)=\alpha(E)q^d + O(q^{d-1}).
\]
Moreover, $\alpha(E)>0$ if and only if $E$ is big. See \autoref{cor.multiplicityOfDirectSummands}.  The ample $F$-signature of $X$ is
\[
\mathcal{a}(X) = \sum_{[E]\in \Cl(X) \cap \ample(X)} \alpha(E)\in [0,1] \cap \bQ
\]
where the sum traverses all ample divisor classes. A direct computation (\cite[\S4]{CarvajalRojasPatakfalviVarietieswithAmpleFrobenius-traceKernel}) yields
\[
\mathcal{a}(\bP^2) =1, \, \mathcal{a}(\bP^1 \times \bP^1) =1, \,   \mathcal{a}(\Bl_{x_0} \bP^2)=1/2, \, \mathcal{a}(\Bl_{x_0,x_1} \bP^2)=0, \, \mathcal{a}(\Bl_{x_0,x_1,x_2} \bP^2)=0.
\]

In \autoref{sec.AmpleFSignHomogeneity}, we find out to what extent the above picture holds in higher dimensions. It can be summarized as follows.

\begin{theoremC*}[{\autoref{cor.PositivityAmpleFSignature}, \autoref{thm.CharacterizationHomogenousSpaces}}]
    With notation as above, the following two statements hold:
    \begin{enumerate}
        \item $\mathcal{a}(X)>0$ if and only if there is a log Fano toric structure $(X,\Delta)$ of class index $1$ (see \autoref{subsubse.LogStructures} for this terminology).
        \item  $\mathcal{a}(X)=1$ if and only if $X$ is a homogeneous space.
    \end{enumerate}
\end{theoremC*}

Although the above Frobenius-theoretic characterization of homogeneity is numerical, this is just on the surface, as it can be written qualitatively as follows. It means that $X$ is a homogeneous space if and only if for all $e>0$ we have that $\sE_{X,e}$ is nef and its big invertible direct summands are ample.

\begin{convention}
 We fix an algebraically closed field $\kay$ of characteristic $p>0$. Unless otherwise stated, all varieties are defined over $\kay$ and are assumed to be normal and projective. We use the shorthand notation $q \coloneqq p^e$ for $e \in \bN$ as well as $q'\coloneqq p^{e'}$, $q_0 \coloneqq p^{e_0}$, etc. Finally, if $M$ is an $R$-module and if $N \subset M$ and $S \subset R$ are any subsets, we let $\langle N \rangle_S \subset M$ denote the subset of linear combinations of elements in $N$ with coefficients in $S$. This notation will come in handy when denoting polytopes and cones defined by a different set of vectors and coefficients.
\end{convention}

\subsection*{Acknowledgements} 
The authors are grateful to Fabio~Bernasconi, Leonid~Monin, Stefano~Filipazzi, Joaqu\'in~Moraga, and Zsolt~Patakfalvi for their help in writing this article.

\section{Preliminaries on the Basic Geometry of Toric Varieties} \label{sec.ToricGeometryPreliminaries}

For the reader's convenience, we survey some basics on toric varieties that will be directly relevant in our proofs.  We refer to \cite{CoxLittleSchenckToricVarieties} for further details and to \cite{MustataToricNotes} for a characteristic-free treatment. We also follow Lazarsfeld's books on positivity \cite{LazarsfeldPositivity1,LazarsfeldPositivity2}. 

We say that a variety $X$ is \emph{toric} if there is an open embedding of the ($d$-dimensional) torus $\bT \coloneqq \bG_{\mathrm{m}}^d \to X$ such that the $\bT$-action of $\bT$ on itself extends to $X$. Of course, $d = \dim X$. As is customary in the literature, we follow the following conventions:

\begin{enumerate}
    \item The \emph{group of characters of $\bT$} is $M \coloneqq \Hom(\bT, \bG_{\mathrm{m}}) \simeq \Z^d$ and the dual lattice of \emph{one parameter subgroups} is $N \coloneqq \Hom(\bG_{\mathrm{m}},\bT) \simeq \bZ^d$. 
    \item A toric variety $X$ is described by a fan $\Sigma=\Sigma_X=\{\sigma_i\}_{i \in I}$ contained in the vector space $N_{\R} \coloneqq N \otimes_{\Z} \R \simeq \bR^d$. We also denote the dual vector space by $M_{\R} \coloneqq M \otimes_{\Z} \R$. A toric variety with fan $\Sigma$ is denoted by $X_{\Sigma}$.
    \item Every $\sigma \in \Sigma$ is a rational (polyhedral) and strongly convex cone. This means that $\sigma = \langle v_1,\ldots,v_m \rangle_{\bR_{\geq 0}} $ for some $v_1,\ldots,v_m \in N$ (rationality) and $\sigma \cap -\sigma = 0$ (strong convexity). A toric variety is $\Q$-factorial if and only if the fan $\Sigma$ is simplicial, i.e., the minimal generators of every $\sigma \in \Sigma$ are linearly independent. Likewise, $X$ is smooth if and only if $\Sigma$ is smooth, i.e., the minimal generators of $\sigma$ extend to a basis of $N$.
    \item The set of $k$-dimensional cones of $\Sigma$ is denoted by $\Sigma(k)$. Set $r \coloneqq |\Sigma(1)|$ and $s \coloneqq |\Sigma(d)|$
    \item For every element $\rho_i$ in $\Sigma_X(1)$, there is a unique element $u_i \in N\cap \rho_i$ that generates $N \cap \rho_i$ as a semigroup. We call $u_i$ the \emph{primitive ray generator} of $\rho_i$.
    \item The orbit-cone correspondence \cite[Theorem 3.2.6]{CoxLittleSchenckToricVarieties} establishes a natural bijection between the elements of $\Sigma_X(k)$ and the $\bT$-invariant subvarieties of $X$ of codimension $k$. In particular, to a one-dimensional cone $\rho$ of $\Sigma_X$ there corresponds a unique $\bT$-invariant prime divisor $P_{\rho}$ on $X$. For notation ease, if $\Sigma_X(1)=\{\rho_1, \ldots, \rho_r\}$, we write $P_{1},\ldots, P_{r}$ for the corresponding list of $\bT$-invariant prime divisors on $X$. 
    Likewise, $X^{\bT} \coloneqq \{x_1,\ldots,x_s\} \subset X$ will denote the list of $\bT$-invariant points---the fixed points by the action of $\bT$. 
\end{enumerate}

\begin{convention}
    For the remainder of this section, we set $X$ to be a (projective) $\bQ$-factorial toric variety of dimension $d$.
\end{convention}

 The morphism $\bZ^r \to N$ given by $e_i \mapsto u_i$ has as its dual the map
\[
\Div \:M \xrightarrow{\chi \mapsto \Div \chi \coloneqq u_1(\chi) P_1 + \cdots + u_r(\chi) P_r } \bZ P_1 \oplus \cdots \oplus \bZ P_r \simeq \bZ^r.
\]
After fixing a $\bZ$-basis for $N$ (and the dual one on $M$) and writing $u_i\in N$ as a column $\bZ$-vector, the map $\Div$ fits into an exact sequence
\begin{equation} \label{eqn.SESPicard}
    0 \to M \simeq \bZ^d \xrightarrow{\Div \simeq [u_1 \cdots u_r]^\top} \bZ^r \to \Cl(X) \to 0
\end{equation}
as $X$ has no torus factor; see \cite[Theorem 4.1.3]{CoxLittleSchenckToricVarieties}. More succinctly, every divisor $D$ on $X$ is linearly equivalent to a $\bT$-invariant divisor $a_1 P_1 + \cdots + a_r P_r$ and this representation is unique modulo divisors of characters. 

On a toric variety, a divisor is numerically trivial if and only if it is linearly torsion. Then, there is a split short exact sequence
\[
0 \to \Cl(X)_{\mathrm{tor}} \to \Cl(X) \xrightarrow{\varpi} N^1(X) \to 0
\]
where $N^1(X) \simeq \bZ^{\rho}$ is the Néron--Severi group of $X$ as in Lazarsfeld's textbooks, i.e., divisors modulo numerical equivalence. We may let $\varsigma \: N^1(X) \to \Cl(X)$ be a chosen section of $\varpi$ so that we may write $\Cl(X) = N^1(X) \oplus  \Cl(X)_{\mathrm{tor}}$.

We will let $\pi_{i} \in N^1(X)$ (for $i=1,\ldots,r$) denote the numerical class of $P_{i}$. 
If $D$ is a divisor on $X$, we denote its divisor class in $\Cl(X)$ by $[D]$ and its numerical class by $\llbracket D \rrbracket \in N^1(X)$. 

Since $\Pic(X)$ is torsion-free (\cite[Propositions 4.2]{CoxLittleSchenckToricVarieties}, \cf \cite[Corollary 5.4]{CarvajalFiniteTorsors}), the canonical homomorphisms
\begin{equation} \label{eqn.DivisorGroups}
\Pic(X) \xrightarrow{\subset} \Cl(X) \twoheadrightarrow N^1(X)
\end{equation}
realize $\Pic(X)$ as a subgroup of $N^1(X)$ (i.e., numerical and linear equivalence coincide on Cartier divisors).\footnote{However, this does not imply that $\Pic(X)$ is inside the free part of $\Cl(X)$; see \cite{RossiTerraciniEMBEDDINGpICARDgROUP}.} Since $X$ is further $\bQ$-factorial, $\Pic(X)$ is free of rank $\rho\coloneqq \rho(X)$---the \emph{Picard rank} of $X$. See \cite[Propositions 4.2.7]{CoxLittleSchenckToricVarieties}. In particular, we have the identity 
\[
\rho = r - d.
\] 
More generally, the canonical homomorphisms in \autoref{eqn.DivisorGroups} are generic isomorphisms (although these are isomorphisms when $X$ is smooth). This common real vector space
\[
\Pic(X)_{\bR} = \Cl(X)_{\bR} = N^1(X)_{\bR} \coloneqq N^1(X) \otimes_{\bZ} \bR \simeq \bR^{\rho}
\]
is the \emph{Néron--Severi space} of $X$. We shall think of $N^1(X)$ as a lattice inside $N^1(X)_{\bR}$ and refer to its elements as the \emph{lattice points} of $N^1(X)_{\bR}$. 

Given a map $\phi\: X \to S$ between $\bQ$-factorial toric varieties, there is an induced pullback $\phi^* \: N^1(S)_{\bR} \to N^1(X)_{\bR}$ by applying $-\otimes_{\bZ} \bR$ to $\phi^{*} \: \Pic (S) \to \Pic (X)$.

Dually, we let $N_1(X)$ denote the abelian group of algebraic 1-cycles on $X$ modulo numerical equivalence and $N_1(X)_{\R} \coloneqq N_1(X) \otimes_{\Z} \R$. We let $\mov(X)$ denote the cone of moving curves in $N_1(X)_{\R}$, which is the dual cone of the pseudo-effective cone $\eff(X)$; see \cite{MovingConeIsDualOfPseudoEffectiveCone}. These can be described as follows (we believe this to be ``well-known among experts'' but we provide a proof for the sake of completeness). 

\begin{proposition}[The pseudo-effective cone] \label{prop.PseudoEffectiveCone}
    With notation as above,
    \[
\eff(X) = \langle \pi_1 , \ldots , \pi_r  \rangle_{\bR_{\geq 0}}.
\]
In particular, $\eff(X)$ is a strongly convex rational cone such that
\[
N^1(X) \cap \eff(X) = \langle \pi_1, \ldots , \pi_r  \rangle_{\bN}.
\]
and $\dim \mov(X) = \rho(X)$.
\end{proposition}
\begin{proof}
    Note that every effective divisor on $X$ is linearly equivalent to an effective $\bT$-invariant divisor as $\Div \chi$ is $\bT$-invariant for all $\chi \in M$. Indeed, let $D\geq 0$ be linearly equivalent to a $\bT$-divisor $D'=a_1 P_1 + \cdots +a_r P_r$. Since $\sO_X(D')$ admits a nonzero section, so does $D'$. This means that there is $\chi$ such that $\Div \chi + D' \geq 0$. In particular, $\Div \chi + D'$ is an effective $\bT$-invariant divisor linearly equivalent to $D$. That is, the second equality displayed holds.

    The above further shows that the effective cone of $X$ is equal to $\langle \pi_1, \ldots ,\pi_r \rangle_{\bR_{\geq 0}}$ and is closed, in particular. This implies that it is equal to the pseudo-effective cone (i.e., the first displayed equality holds) and that $\eff(X)$ is a rational (polyhedral) cone.

    All that is left to explain is why $\eff(X)$ is strongly convex. This seems to be a general feature of pseudo-effective cones on smooth projective varieties; see \cite[Lemma 2.3]{CasciniHaconMustataSchwedeNumericalDimensionOfPseudoEffectiveConesPosChar}. In our case, we provide a simple proof in our case. Note that
    \[
    \Div(M) \cap \langle P_1,\ldots,P_r \rangle_{\bN} = 0
    \]
    as $H^0(X,\sO_X)=\kay$.\footnote{Alternatively, this is to say that an effective divisor whose inverse has a nonzero section must be zero.} Similarly, $\Div(M)$ only intersects $-\langle P_1,\ldots,P_r \rangle_{\bN}$ at zero. By clearing the denominators, we see that
\[
\Div_{\bQ}(M_{\bQ}) \cap \langle P_1,\ldots,P_r \rangle_{\bQ_{\geq0}} = 0.
    \]
In other words, the linear $\bQ$-subspace $V\coloneqq \Div_{\bQ}(M_{\bQ}) \subset \bQ P_1 \oplus \cdots \oplus \bQ P_r = \bQ^r$ avoids the first orthant (except for the origin). Say $V \subset (\bQ^r \setminus O) \cup 0$ where $O$ is the said orthant. Completing with respect to the standard norm (or, say, taking euclidean closures inside $\bR^r$) yields that $V_{\bR} = \Div_{\bR}(M_{\bR})$ avoids the \emph{interior} of the first orthant of $\bR^r$. Furthermore:
\begin{claim}
   $V_{\bR}$ only intersects the boundary of the first orthant of $\bR^r$ at the origin.
\end{claim}
\begin{proof}[Proof of claim] We do induction on $r$. If $r=1$ then $V=0$ and there is nothing to prove (likewise for $r=2$). For the inductive step, consider $V_{\bR} \cap H_i$  for each of the standard hyperplanes $H_i \subset \bR^r$, say $H_i \: x_i=0$. Since $V \cap \{x_i = 0\}$ avoids the first orthant of $\{x_i=0\} \simeq \bQ^{r-1}$, the inductive hypothesis implies that $V_{\bR} \cap H_i$ only intersects the first orthant of $H_i$ at the origin. Since $i$ is arbitrary, the claim follows.
\end{proof}
The claim means that $\Div_{\bR}(M_{\bR}) \cap \langle P_1,\ldots,P_r \rangle_{\bR_{\geq0}} = 0$, hence $\eff(X)$ is strongly convex.
\end{proof}

We will also need the following description of the interior of $\eff(X)$, i.e. the \emph{big cone} of $X$.

\begin{corollary}[The big cone] \label{cor.TheBigCone}
With notation as above, the big cone of $X$ is given by
\[
\bigc(X) = \langle \pi_1, \ldots , \pi_r  \rangle_{\bR_{> 0}}.
\]
\end{corollary}
\begin{proof}
\autoref{prop.PseudoEffectiveCone} can be rephrased by saying that we have an exact sequence
\[
0 \to \bR^d \xrightarrow{\bR \otimes \Div} \bR^r \xrightarrow{\varpi} \bR^{\rho} \to 0
\]
and $\varpi(O) = \eff(X)$, where $O \subset \bR^r$ is the first orthant. We further claim that
\[
\varpi(O^{
\circ
}) = \eff(X)^{\circ} = \bigc(X).
\]
To see this equality, note that $\varpi(O^{\circ}) \subset \bigc(X)$ as $\varpi$ is open----it is a surjective linear transformation. In particular,
\[
\varpi^{-1}\big(\bigc(X)\big) \cap O \supset O^{\circ}
\]
and the equality is achieved as $\varpi$ is continuous. Thus,
\[
\varpi^{-1}\big(\bigc(X)\big) \cap O = O^{\circ}.
\]
The remaining inclusion $\varpi(O^{\circ}) \supset \bigc(X)$ then follows. Indeed, if $\beta \in \bigc(X)$, there is $\alpha \in O$ such that $\varpi(\alpha)=\beta$ and so $\alpha \in O^{\circ}$.
\end{proof}

\subsection{On the Mori geometry of toric varieties}
Tensoring and dualizing \autoref{eqn.SESPicard} yields
\[
0 \to N_1(X)_{\R} \to \R^r \xrightarrow{e_i \mapsto u_i} N_{\R} \to 0
\]
which implies that $N_1(X)_{\R}$ can be realized as the space of $\bR$-linear relations among the primitive ray generators of $\Sigma_X$ in $N_{\bR}$. Explicitly, a relation
    \[
        a_1 u_1+\cdots +a_r u_r = 0, \quad a_i \in \R
    \]
    corresponds to the class of the $\bR$-linear $1$-cycle that has intersection $a_i$ with $P_i$ for all $i \in \{1,\dots,r\}$. In particular, we will think of relations as such as elements in $N_1(X)_{\R}$. 

\subsubsection{Primitive relations and the Mori cone}
Among such relations between the primitive ray generators, there are special ones that completely determine the Mori cone of $X$. Namely, \emph{Batyrev's primitive relations}, which we will explain next. 
\begin{definition}[{Primitive collections \cite{BatyrevPrimitiveCollection}}]\label{defPrimitiveRel}
    A nonempty subset $\sP \subset \Sigma_X(1)$ is called a \emph{primitive collection} if $\sP$ does not span a cone in $\Sigma_X$ but every proper subset of $\sP$ does. 
  \end{definition}  
    Let $\sP=\{\rho_1,\ldots, \rho_k\}$ be a primitive collection on $X$ and $v \coloneqq u_1+\cdots+u_k$. By completeness, there is a cone $\sigma' \in \Sigma_X$ such that $v \in \sigma'$. We may take $\sigma'$ of minimal dimension, which Batyrev calls the \emph{focus} of $\sP$. 
    
    If $X$ is smooth, every set of primitive ray generators in $\sigma'$ is disjoint from $\{u_1,\ldots, u_k\}$; see \cite[Proposition 3.1]{BatyrevPrimitiveCollection} which uses the fact that $\sP$ is a primitive collection. Let $u_{k+1},\ldots, u_{l}$ be the primitive ray generators in $\sigma'$. Then, we may write $v = a_{k+1} u_{k+1}+\cdots + a_l u_l$ for some unique coefficients $a_{k+1},\ldots,a_l \in \bN \setminus \{0\}$, possibly with $l=k$. Hence, we obtain a relation
\[  
R_{\sP} \: u_1 + \dots + u_k - a_{k+1}u_{k+1} - \dots - a_l u_{l} = 0,    
\] 
which is referred to as the \emph{primitive relation} defined by $\sP$ and we treat it as an element in $N_1(X)_{\bR}$. When $X$ is not smooth, there might be some $u_j$ appearing in the above relation with positive and negative coefficients. However, we may rewrite the above relation as
    \[
        R_{\sP} \: b_1u_1 + \dots + b_ku_k - b_{k+1}u_{k+1} - \dots - b_{l}u_{l} = 0,
    \]
    where all the coefficients are positive natural numbers with no common factors. See \cite[Exercise 6.4.3]{CoxLittleSchenckToricVarieties}. The associated \emph{degree} of the primitive collection/relation is defined by 
    \[
    \deg \sP \coloneqq \deg R_{\sP} \coloneqq -K_X \cdot R_{\mathcal{P}} = k - (a_{1}+\cdots +a_l) = \sum_{i=1}^kb_i - \sum_{j=k+1}^lb_{j} \in \bZ.
    \]

The following result underlines the importance of primitive relations. 
\begin{theorem}[The Mori Cone {\cite[Theorem 6.4.11]{CoxLittleSchenckToricVarieties}}]
With notation as above, the Mori cone of $X$ is given by
    \[
        \moriCone(X) = \langle R_{\sP} \mid \sP \text{ is a primitive collection} \rangle_{\bR_{\geq 0}}. 
    \] 
    In particular, $\moriCone(X)$ is a strongly convex rational cone.
\end{theorem}

\begin{remark}
    We do not need the above result to see that $\moriCone(X)$ is a rational strongly convex cone. Indeed, $\moriCone(X)$ is spanned by the (classes of) $\bT$-invariant curves of $X$ (which correspond to the walls in $\Sigma_X(d-1)$) and its dual cone has maximal dimension, the Picard rank. 
\end{remark}

\begin{remark}[Extremal primitive relations]
    The authors do not know of any characterization or criterion to tell exactly which $R_{\sP}$ generate the extremal rays of $\moriCone(X)$. However, Casagrande gave a sufficient condition for it in \cite[Proposition 4.3]{CasagrandeContractibleClassesInToricVarieties}. Contrast this with the Ph.D. thesis of Monsôres, where a characterization is provided for extremal rays of $\mov(X)$; see \cite[Theorem 5.3.3]{MonsoresPhDThesis}. We refer to those $R_{\sP}$ that generate extremal rays of $\moriCone(X)$ as \emph{extremal primitive relations}.
\end{remark}

\subsubsection{Toric extremal contractions}\label{subsection.ToricExtremalContractions} In order to describe the geometry of extremal contractions associated with primitive relations, we first need to recall the \emph{secondary fan} $\Sigma_{\mathrm{GKZ}}$ of $X$. 

Recall that to a $\bT$-invariant divisor $D$ on $X$ one attaches a polytope $P_D$. The normal fan of $P_D$ corresponds to a toric variety $X_D$ that reflects the properties of $D$. For example, $X_D \simeq X$ if and only if $D$ is ample. See \cite[6.2]{CoxLittleSchenckToricVarieties}. The secondary fan $\Sigma_{\text{GKZ}}$ of $X$ is a decomposition of $\eff(X)$ into polyhedral cones $\sigma$ such that the construction $[D] \mapsto X_D$ is constant in the relative interior of $\sigma$. The maximal dimensional cones of $\Sigma_{\text{GKZ}}$ are called the \emph{chambers of the secondary fan}. For example, the nef cone $\nef(X)$ of $X$ is a chamber of $\Sigma_{\text{GKZ}}$.

By the duality between the cones $\moriCone(X)$ and $\nef(X)$, an extremal ray $R$ of $\moriCone(X)$, which is identified with an extremal primitive relation $R_{\sP}$, corresponds to a facet $F_R$ of $\nef(X)$. Now, $F_R$ may or may not be on a facet of $\eff(X)$. That is, when one wall crosses $F_R$ from the interior of $\nef(X)$, one may or may not end up in the interior of $\eff(X)$. If $F_R$ is not a facet of $\eff(X)$, it is a divisorial wall or a flipping wall. Thus, the following trichotomy arises:
 
   \begin{itemize}
    \item Suppose that $F_R$ is on the boundary of $\eff(X)$. Then, there is no chamber on the other side of $F_R$ and we obtain a fibration $\phi_R \: X \rightarrow X_D$ where $[D]$ is in the relative interior of $F_R$. We refer to them as \emph{Mori fibrations}.
   \item  Suppose that $F_R$ is not on the boundary of $\eff(X)$ and let $\sigma$ be the chamber of $\Sigma_{\text{GKZ}}$ lying on the other side of $F_R$. Then, the wall crossing yields a birational morphism $\phi_R \:X \rightarrow X_{D}$ where $[D]$ is an element of the relative interior of $F_R$.
   \begin{itemize}
       \item  If $F_R$ is a divisorial wall, then $\phi_R$ contracts a prime $\bT$-invariant divisor, and the Picard rank decreases by $1$. We refer to them as \emph{extremal divisorial contractions}.
        \item If $F_R$ is a flipping wall, $\phi_R$ is a small contraction, and it induces an isomorphism between the Néron--Severi spaces, thereby preserving the Picard rank. The flip of the small contraction $\phi_R$ is given by a birational map $f_R \: X \dashrightarrow X_E$ where this time $[E]$ is an element in the relative interior of $\sigma$.
   \end{itemize}
\end{itemize}

In either case, the map $\phi_R \: X \rightarrow X_D$ is the extremal contraction corresponding to the extremal ray $R$. The flip $f_R$ is an example of a \textit{small $\Q$-factorial modification}, that is, $f_R$ is an isomorphism in codimension $1$ and $X_{D}$ is $\Q$-factorial.  See \cite[Ch. 15]{CoxLittleSchenckToricVarieties} for details. What we will take advantage of is that the properties of $\phi_R \: X \rightarrow X_D$ are determined by the corresponding primitive relation $R_{\sP}$ as established in the following result.

\begin{theorem}\label{thm.MoriTheoryOfFanoVarieties}
    With notation as above, let
    \[
        R_{\sP} : b_1u_1 + \dots + b_ku_k - b_{k+1}u_{k+1} - \dots - b_lu_l = 0,
    \]
    be an extremal primitive relation, i.e, it spans an extremal ray $R$ of $\moriCone(X)$. Let $\phi_R \: X\to S$ be the associated extremal contraction. The following statements hold:
    \begin{enumerate}
    \item  The positive dimensional fibers of $\phi_R$ have dimension $k-1$. 
        \item The codimension of the exceptional locus of $\phi_R$ is $l-k$. In particular:
        \begin{enumerate}
            \item $\phi_R$ is a Mori fibration if and only if $l=k$.
            \item $\phi_R$ is divisorial if and only if $l=k+1$. In that case, $P_{k+1}$ is the exceptional divisor, and $P_{k+1} \cdot R_{\mathcal{P}} = -b_{k+1}$.
            \item $\phi_R$ is small if and only if $l\geq k+2$. 
        \end{enumerate}
        \item If $X$ is smooth, then $b_1 = \cdots = b_k =1$ and $\ell(R) = k - (b_{k+1} + \dots + b_l)$.
    \end{enumerate}
\end{theorem}
\begin{proof}
    For the proofs of (a) and (b), see \cite[Proposition 15.4.5]{CoxLittleSchenckToricVarieties}. For (c), it remains to explain why $\el(R) = \deg R_{\sP}$. To this end, note that since $X$ is smooth there exists a coefficient of $R_{\sP}$ equal to $1$ and it is primitive in $R$. Then $R_{\sP}$ coincides with the class of the rational curve generating $R$.  
\end{proof}

\subsubsection{Toric extremal Fano varieties}\label{rem.BoundsForLengths}
    Suppose that $X$ is smooth. With notation as in \autoref{thm.MoriTheoryOfFanoVarieties}, let $d_R$ be the maximal dimension of a fiber of $\phi_R$ as in \autoref{sec.Intro}. Then we see that $d_R +1 = \ell(R)$ if $R$ is a fibration and $d_R \geq \ell(R)$ if $\phi_R$ is birational.\footnote{The Ionescu--Wiśniewski inequalities for toric varieties are also verified.} Moreover, $\ell(R) = d_R$ if and only if $\phi_R$ is a divisorial contraction whose exceptional divisor intersects $R_{\mathcal{P}}$ with value $-1$. From this we may conclude that all toric Mori fibrations are projective bundles. Indeed, the $\bC$-model of a toric Mori fibration $\phi_R \: X \to S$ must be a projective bundle by \cite[Theorem 1.3]{HoringNovelliMoriContractionsOfMaximalLength}. Since this is something that depends only on the combinatorics of the polytopes involved, it is characteristic-free. It is worth noting that a direct characteristic-free proof of this fact can be worked out directly. See, for example, \cite[Proposition 3.3.8]{MonsoresPhDThesis}, which further shows that $X=\bP(\sF)$ where $\sF$ is split as a direct sum of invertible sheaves on $S$. Likewise, we may conclude that a birational extremal contraction such that $\ell(R) = d_R$ (i.e., with maximal length) must be a smooth blowup by \cite[Theorem 5.1]{AndreattaOcchettaSpecialRaysInTheMoriConeOfaProjVar}. 

\begin{corollary}
    A smooth toric variety is an extremal Fano variety if and only if every birational extremal contraction is a smooth blowup.
\end{corollary}

\subsubsection{Centrally symmetric primitive relations} \label{Rem.CentrallySymmetricPrimitiveRelations}
Let $X$ be a smooth $d$-dimensional toric variety. By \cite[Proposition 3.2]{BatyrevPrimitiveCollection}, there is a primitive collection $\sP$ with zero focus, i.e., whose primitive relation is of the form $R_{\sP}\: u_1 +\cdots + u_k = 0$. These types of relations are referred to in the literature as \emph{centrally symmetric primitive relations}. Since $\sP$ is a primitive collection, any proper subset of $\{u_1,\ldots,u_k\}$ spans a cone in $\Sigma$. However, since $\Sigma$ is simplicial, a cone is generated by at most $d$ elements, so $k \leq d+1$. 

In \cite{MinimalPBundleDim}, the authors coined the term \emph{minimal projective bundle dimension} for the minimal degree of a centrally symmetric primitive relation minus $1$. Thus, the minimal projective bundle dimension $m(X)$ belongs to $\{1,\ldots,d\}$. This is very much inspired by \cite{ChenFuHwangMinimalRationalCurvesOnCompleteToricManifolds}, where it is established that centrally symmetric primitive relations correspond to \emph{minimal dominating families of rational curves on $X$}. More precisely, although centrally symmetric primitive relations might not be extremal,\footnote{By \cite[Proposition 4.1]{BatyrevPrimitiveCollection}, this is the case if and only if $\sP$ is disjoint from any other primitive collection.} they define generic projective bundles. That is, from a centrally symmetric relation of degree $k$ one constructs a $\bP^{k-1}$-bundle $U \to S$ where $U \subset X$ is an open subset. However, what matters to us is the much simpler statement that $X \simeq \bP^d$ if and only if $d=m(X)=k-1$. Indeed, if $X$ were to have a primitive collection of cardinality $d+1$, then it would have the fan of a projective space of dimension $d$. In conclusion, we obtain the following characterization of projective spaces among smooth toric varieties. This will play a crucial role in our proof of \autoref{thm.MainTheoremBignessAmplenessProjSpace}.

\begin{proposition}[Characterization of projective spaces]\label{prop.ToricCharacterizationOfProjectiveSpace}
    A smooth toric variety is a projective space if and only if its minimal projective bundle dimension coincides with its dimension. In other words, a smooth toric variety $X$ is a projective space if and only if it admits a centrally symmetric primitive relation of degree $\dim X+1$.
\end{proposition}

\subsubsection{Toric nef and moving cones} To conclude our remarks on the Mori geometry of toric varieties, we recall how to obtain the nef cone and the moving cone of divisors of a toric variety. It should be noted that the following description can be extended to the other chambers of the secondary fan. 
\begin{proposition}[{\cite[Proposition 15.2.1]{CoxLittleSchenckToricVarieties}}]\label{proposition.nefConeContainedInSomeCone}
    For each $x \in X^{\bT}$, the set $J_{x} \coloneqq \{ \pi_i \mid P_i \notin x\}$ is a basis of $N^1(X)_{\R}$ such that $\langle J_{x}\rangle_{\bR_{\geq 0}} \supset \nef(X)$. Conversely, any such basis is of this form and moreover
    \[
        \nef(X) = \bigcap_{x \in X^{\bT}} \langle J_{x}\rangle_{\bR_{\geq 0}}.
    \]
    In particular, a facet of $\nef(X)$ is contained in a facet of $\langle J_{x}\rangle_{\bR_{\geq 0}}$ for some $x \in X^{\bT}$.
\end{proposition}
\begin{proof}
    It only remains to explain the last claim, as the rest is stated in \cite[Proposition 15.2.1]{CoxLittleSchenckToricVarieties}. Notice that the relative interiors of the cones $\langle J_{x}\rangle_{\R \geq 0}$ for different points $x \in X^{\bT}$ are not disjoint, as $\nef(X)$ is maximal dimensional. The claim then follows from the following general principle. Let $C=C_1 \cap \cdots \cap C_n$ be a polyhedral cone obtained as the intersection of polyhedral cones $C_i$ such that the relative interiors of $C_j$ and $C_k$ intersect for all $j,k$. Then, a face $F$ of $C$ is obtained as an intersection $F=F_1 \cap \cdots \cap F_n$ where the $F_i$ is a face of $C_i$; see \cite[2.4, Exercise 9, (iv)]{GrünbaumConvexPolytopes}. In case $F$ is a facet, then the face $F_i \subset C_i$ must be either a facet or $C_i$ itself. However, not all of the $F_i$ can be equal to $C_i$.
\end{proof}

From this we obtain the following result that will be instrumental in our proofs in \autoref{sec.Nefness}.

\begin{corollary}\label{keyLemma}
    If $R$ is an extremal primitive relation, then $ \dim_{\bR} \langle\pi_i \mid \pi_i \cdot R = 0 \rangle_{\bR} = \rho -1$.
\end{corollary}
\begin{proof}
By \autoref{proposition.nefConeContainedInSomeCone}, the facet $F_R$ associated with $R$ is contained in a facet of $\langle J_{x}\rangle_{\bR_{\geq 0}}$ for some $x \in X^{\bT}$. This facet is generated by $\rho  - 1$ linearly independent classes $\pi_i \in J_{x}$ and their intersection number with $R$ is zero.
\end{proof}

On a toric variety, nef divisors are necessarily globally generated, i.e., base-point-free. A natural weakening of this condition is that the base locus has codimension $\geq 2$. These are the so-called \emph{moving divisors}. They span the \emph{moving cone of divisors} of $X$, which is denoted as 
\[
\mov^1(X) \subset N^1(X)_{\bR}.
\]
We say that a divisor \emph{moves} if its numerical class belongs to this cone. The moving cone of divisors also admits a nice description in terms of the fan of $X$. 

\begin{proposition}[{\cite[Proposition 15.2.4, Theorem 15.1.10]{CoxLittleSchenckToricVarieties}}] \label{prop.MovingConeToricVarieties}
    With notation as above,
    \[
        \mov^1(X) = \bigcap_{i=1}^r \langle \pi_1,\dots,\pi_{i-1},\pi_{i+1},\dots,\pi_r\rangle_{\R \geq 0} = \bigcup_i f_i^*(\nef(X_i))  \subset N^1(X)_{\R},
    \]
    where the $f_i: X \dashrightarrow X_i$ are the finitely many toric small $\bQ$-factorial modifications of $X$. Moreover, each $f_i^*(\nef(X_i))$ is a chamber of the secondary fan $\Sigma_{\mathrm{GKZ}}$.
\end{proposition}

\begin{remark} \label{rem.NefConesOfQFactorialModifications}
   Toric varieties are examples of Mori dream spaces, which satisfy the above decomposition of the cone of moving divisors. To better understand the maps $f_i$ in the above proposition, note that an extremal ray $R$ corresponds to an extremal contraction that is not small if and only if $F_R$ is on the boundary of $\movd(X)$. Following the above description of flips for toric varieties, we see that $f_i$ is a composition of flips. In particular, although the $X_i$ are uniquely determined by the chambers of $\Sigma_{\mathrm{GKZ}}$ inside the moving cone of divisors, the $f_i$ are not. The reason is that they could be expressed as the composition of toric flips in multiple ways, i.e., there could be many paths in which one can jump from chamber to chamber inside the moving cone to go from one to another.

   To avoid this ambiguity in the notation, we can do the following. Note that for a small $\Q$-factorial modification $f_i \: X\dashrightarrow X_i$, we have $\Sigma_X(1) = \Sigma_{X_i}(1)$. Therefore, the relations between the primitive ray generators of $\Sigma_X$ and $\Sigma_{X_i}$ are the same. This is to say that there is a canonical isomorphism between $N_1(X)$ and $N_1(X_i)$, i.e., all the $f_i$ induce the same isomorphism under pullback. In particular, we get a canonical isomorphism between $N^1(X)$ and $N^1(X_i)$. This can be made very explicit. Indeed, we can identify the $\mathbb{T}$-invariant prime divisors of $X_{i}$ with those of $X$ and see that the linear and numerical relations between them are the same, as they share a common big open subset. Moreover, this also shows that they share the same pseudo-effective cone and so the same big cone. 
   
   In summary, under the above identifications, we can say that the toric $\bQ$-factorial modifications $X=X_0, X_1,\ldots, X_k$ share a common Néron--Severi space $N^1(X)_{\bR}$ and the same pseudo-effective cone inside it. Moreover, they share a common moving cone of divisors, which admits the following decomposition into chambers of $\Sigma_{\mathrm{GKZ}}$
   \[
    \mov^1(X) = \nef(X_0) \cup \nef(X_1) \cup \cdots \cup \nef(X_n).
   \]
   This lets us understand the facets of the moving cone of divisors as follows. Each such facet corresponds to a non-small extremal contraction of a $\bQ$-factorial modification of $X$. It corresponds to a fibration if and only if it sits inside a facet of $\eff(X)$.
\end{remark}

\subsection{On the local structure of extremal divisorial contractions and singularities}  Let us commence by describing the local structure of $\bQ$-factorial toric varieties.
\subsubsection{On the singularities of $\bQ$-factorial toric varieties}\label{subsec.QFactorialToricSingularities} Let $X$ be a $\bQ$-factorial toric variety. It admits an open covering by toric affine subvarieties $U\subset X$. We may shrink $U$ further if necessary to ensure that $\Pic (U) = 0$. Then, on $U$ we would have $\rho=0$, $r=d$, resulting in $\Cl(U)$ being a finite group. After relabeling if necessary, we may say that $P_1,\ldots,P_d$ are the only $\bT$-invariant prime divisors of $X$ whose generic point is in $U$. Hence, we may think of them as the $\bT$-invariant prime divisors of $U$ and the restriction map gives an exact sequence
\[
0 \to \langle \pi_{d+1},\ldots, \pi_{d+\rho} \rangle_{\bZ} \to \Cl(X) \xrightarrow{ \pi_i \mapsto \pi_i} \Cl(U) \to 0 
\]
We will refer to any such $U$ as a \emph{purely $\bQ$-factorial} affine toric chart of $X$. 

For example, for each of the $\bT$-invariant points $x_1,\ldots,x_s \in X$, we may define a purely $\bQ$-factorial affine toric neighborhood $U_i \ni x_i$. The fan of $U_i$ is the fan of subcones of the $d$-dimensional cone $\sigma_i$ corresponding to $x_i$. In this case, the kernel of $\Cl(X) \to \Cl(U)$ is $\langle J_{i}\rangle_{\bZ}$ with $J_i \coloneqq J_{x_i}$ as in \autoref{proposition.nefConeContainedInSomeCone}. Equivalently, $\sigma_i$ is minimally generated by the $d$ primitive ray generators $u_j$ for which $P_j\ni x_i$. In this way, we produce a purely $\bQ$-factorial affine toric open covering
\[
X = U_1 \cup \cdots \cup U_s
\]
and we refer to it as the \emph{standard open covering} of $X$.

Fortunately, we have a good understanding of what an such $U$ is like. Indeed, it turns out that there is a finite $G$-quasitorsor cover 
\[
f\:\bA^d=\Spec \kay[t_1,\ldots,t_d] \to U
\] such that $f^* P_i = \Div t_i$ for all $i=1,\ldots, d$. Here, $G\coloneqq \mathfrak{D}(\Cl(U))$ is the diagonalizable finite algebraic group defined by the finite abelian group $\Cl(X)$. In general, $\mathfrak{D}(-)$ is an exact contravariant functor from the category of abelian groups to the one of group-schemes. It is given by $D(\Gamma)=\Spec \kay[\Gamma]$ where $\kay[\Gamma]$ is the group-algebra of $\Gamma$. For instance, \[
D(\bZ^r \oplus \bZ/m_1 \oplus \cdots \oplus \bZ/m_s)=\bG_{\mathrm{m}}^r \times \bm{\mu}_{m_1}  \times \cdots \times \bm{\mu}_{m_s}.\]
More generally, $\mathfrak{D}$ establishes an anti-equivalence from the category of finitely-generated abelian groups to the one of diagonalizable algebraic groups---those that are subgroups of tori. The adjoint functor of $\mathfrak{D}$ is the one of characters $\Hom(-,\bG_{\mathrm{m}})$. See \cite[Ch. 12]{MilneAlgebraicGroups} for more.

The definition and functioning of the finite cover $f$ above is rather simple. We explain next how it works, but we recommend seeing \cite{LinearlyReductiveQuotientSingularities} or \cite[\S3.1]{CarvajalRojasFayolleTameRamificationCentersOfPurity} for further details.

In general, the action of $\mathfrak{D}(\Gamma)$ on a $\kay$-scheme $S$ amounts to a $\Gamma$-grading 
\[
\sO_S = \bigoplus_{\gamma \in \Gamma} \sF_{\gamma}
\] as a sheaf of $\kay$-algebras. Assume that $\Gamma$ is a finite group. Then $\sF_0 = \sO_S^G \subset \sO_S$ is an integral extension of sheaves of $\kay$-algebras. It is finite precisely when $\sF_{\gamma}$ is a sheaf of finitely generated $\sF_0$-modules for all $\gamma \in \Gamma$. 

For the quotient map $S \to S/\mathfrak{D}(\Gamma)$ to exist, we require the standard condition that every point of $S$ admits an affine open neighborhood containing its $\mathfrak{D}(\Gamma)$-orbit. This translates to the sheaves being $\sF_{\gamma}$ (quasi-)coherent in the following sense. We ask $S$ to admit an affine open covering by affine opens $V \subset S$ such that there are sections $s_1,\ldots,s_n \in \sF_0(V)\subset \sO_S(V)$ such that $\langle s_1,\ldots, s_n \rangle \sO_S(V) = \sO_S(V)$ and the canonical homomorphisms
\[
\sF_{\gamma}(V)_{s_i} \coloneqq \sF_{\gamma}(V) \otimes_{\sF_0(V)} \sF_0(V)_{s_i} \to \sF_{\gamma}(D(s_i))
\]
are isomorphisms for all $i=1,\ldots,n$ and all $\gamma \in \Gamma$. In this case, we obtain a finite $\mathfrak{D}(\Gamma)$-quotient $g\: S \to S/\mathfrak{D}(\Gamma)$ where $g^{\#}$ is the inclusion $\sO_{S/\mathfrak{D}(\Gamma)} = \sF_0  \subset \sO_S$. This morphism is split by a trace map $\Tr_g \: \sO_S \to \sO_{S/\mathfrak{D}(\Gamma)}$, which is the canonical projection onto the zeroth degree summand (see \cite[\S3]{CarvajalFiniteTorsors}). Moreover, $g$ is a $\mathfrak{D}(\Gamma)$-torsor over a point $x \in S/\mathfrak{D}(\Gamma)$ if and only if $(\sF_{\gamma})_x$ is a free $\sO_{S/\mathfrak{D}(\Gamma),x}$-module of rank $1$ for all $\gamma \in \Gamma$. The $\mathfrak{D}(\Gamma)$-torsor locus is an open subset of $S/\mathfrak{D}(\Gamma)$ and when it is big we say that $g$ is a $\mathfrak{D}(\Gamma)$-quasitorsor. 

Let us return to the example that concerns us for now. To construct $f$ above, we must give a $\Cl(X)$-grading of $A\coloneqq \kay[t_1,\ldots,t_d]$. First, observe that the $\bT$-action on $\bA^d$ (where $\bT=\mathfrak{D}(\bZ^d)$) corresponds to the standard $\bZ^d$-grading 
\[
A = \bigoplus_{(n_1,\ldots,n_d) \in \bN^d}  \kay t_1^{n_1} \cdots t_d^{n_d}.
\]
In our case, the exact sequence \autoref{eqn.SESPicard} leads to a presentation
\[
0 \to M\simeq \bZ^d \xrightarrow{\Div = [u_1 \cdots u_d]^{\top}} \bZ^d \xrightarrow{\varpi} \Cl(U) \to 0.
\]
Its dual is an exact sequence
\[
0 \to \bZ^d \xrightarrow{ \Upsilon \coloneqq \Div^{\vee}  = [u_1 \cdots u_d] } \bZ^d \simeq N \xrightarrow{} \Cl(U) \to 0.
\]
In other words,
\[
\Cl(U) = N/\langle u_1,\ldots,u_k \rangle_{\bZ}.
\]
This is transformed by $\mathfrak{D}(-)$ into the exact sequence
\[
0 \to G \to \bT \xrightarrow{D(\Upsilon)} \bT \to 0.
\]
 Since the action of $G$ on $\bA^d$ must be the restriction of the standard toric one, it corresponds to the $\Cl(U)$-grading $A = \bigoplus_{\delta \in \Cl(U)} A_{\delta}$ where
\[
A_{\delta} \coloneqq  \bigoplus_{ (n_1,\ldots,n_d) \in \varpi^{-1}(\delta) \cap \bN^d }  \kay t_1^{n_1} \cdots t_d^{n_d}.
\]
Therefore,
\[
  A^G = A_0 = \bigoplus_{ (n_1,\ldots,n_d) \in \Div(\bZ^d) \cap \bN^d }  \kay t_1^{n_1} \cdots t_d^{n_d} \xrightarrow{\simeq} \sO_U(U),
\]
where the displayed isomorphism of $\kay$-algebras is the one that sends $t_1^{n_1}\cdots t_d^{n_d}$ to the only section $s\: U \to \bA^1$ that restricts to the character $s_{\bG_{\mathrm{m}}} = \chi \: \bT \to \bG_{\mathrm{m}}$ given by 
\[
\chi \coloneqq \mathfrak{D}\Big(\bZ \xrightarrow{1 \mapsto (n_1,\ldots,n_d)} \bZ^d\Big).
\]

This construction also reveals why $f\: \bA^d \to U$ is compatible with the toric structures. First of all, since $G$ acts on $\bA^d$ via the standard action of $\bT$, we find that $\bT=\bT/G$ acts on $U$ and so $U$ is a $\bT$-variety. Furthermore, the toric structure $\bT \subset \bA^d$ is $G$-equivariant and its quotient is the toric structure $\bT = \bT/G \subset U$. In fact, we obtain the cartesian diagram
\[
\xymatrix{
 \bT \ar[r]^-{\subset} \ar[d]_-{f_{\bT} =D(\Upsilon)} & \bA^d \ar[d]^-{f} \\
\bT \ar[r]^-{\subset} & U
}
\]

We can readily see from this what the fan of $\bA^d/G$ is. Recall that the standard fan of $\bA^d$ is $\{ \sigma_I \coloneqq \langle e_i \mid i\in I\rangle_{\bR_{\geq 0}} \}_{I \subset [d]}$ where the $e_1,\ldots,e_d \in \bZ^d$ is the standard basis and $[d] \coloneqq \{1,\ldots,d\}$. Then the fan of $U$ is
\[
\Sigma_U =\big \{\Upsilon(\sigma_I) =  \langle u_i \mid i\in I\rangle_{\bR_{\geq 0}} \big\}_{I \subset [d]}.
\]

It further follows from the construction that the $G$-torsor locus of $f$ is the factorial/regular locus of $U$ and therefore it is big. Moreover, $\sO_U(-P_i) A = (t_i)$ and $f^* P_i = \Div t_i$.

The above describes how to recover $U$ from an $G$-action on $\bA^d$. One can do this backwards too by writing $f$ as the spectrum of an extension of rings
\[
\sO_U \subset \Cox(U) \coloneqq \bigoplus_{[D] \in \Cl(U)} \sO_U(D)
\]
inside the function field $\kay(t_1,\ldots,t_d)$, but we skip the details (see instead \cite[\S3.1]{CarvajalRojasFayolleTameRamificationCentersOfPurity}).

There is one more general principle that we want to extract from this. In general, if $X$ is a toric variety and we have a finite (diagonalizable) subgroup $G\subset \bT$ such that $\bT/G = \bT$ (i.e., we have a $G$-torsor $\bT \to \bT$) then the action of $G$ on $X$ induces a well-defined toric quotient $f\: X \to X/G$ that restricts to the original $G$-torsor $f_{\bT} \: \bT \to \bT$ on the torus. We will refer to such subgroups $G\subset \bT$ as \emph{toric}. Toric subgroups correspond to square $\bZ$-matrices $\Lambda$ of size $d$ and non-zero determinant, say $G=\ker (D(\Lambda))=D(\coker \Lambda)$ and $\Lambda = \Hom(\bT \to \bT/G,\bG_{\mathrm{m}})$, so $|G|=|\det \Lambda|$. Further, 
\[
\Sigma_{X/G} = \{\Lambda(\sigma) \mid \sigma \in \Sigma_X \}.
\]

\subsubsection{The local structure of extremal divisorial contractions} 
The following establishes that toric extremal divisorial contractions are locally ``fake weighted blowups.'' 

\begin{proposition} \label{prop.LocalStructureExtremalDivContraction}
     Let $X$ be a $\bQ$-factorial toric variety and $\phi\: X \to S$ be an extremal divisorial contraction given by an extremal primitive relation
    \[
        R \: b_1u_1 + \dots + b_ku_k - b_{k+1}u_{k+1} = 0.
    \]
    Let $C \coloneqq \phi(P_{k+1}) \subset S$ be the center of $\phi$ and $w=(w_1,\ldots,w_k)\in \bN^k$ be the primitive ray generator of the ray spanned by  $(b_1/b_{k+1},\ldots,b_k/b_{k+1}) \in \bQ^k_{>0}$. Then, there is an open covering of $S$ by purely $\bQ$-factorial toric affine charts $f\: \bA^d \to U$ (as in \autoref{subsec.QFactorialToricSingularities}) such that there is a commutative diagram
        \[
        \xymatrix{
       H \coloneqq C_U \times_U \bA^d = V(t_1,\ldots, t_k) \ar@{^{(}->}[r] \ar[d]_-{f} & \bA^d = \Spec \kay[t_1,\ldots,t_d] \ar[d]_-{f} & \Bl_H^{\mathrm{w}} \bA^d \ar[l]_-{\beta} \ar[d]^-{g} \\
        C_U \ar@{^{(}->}[r]   & U = \bA^d/G & X_U = (\Bl_H^{\mathrm{w}} \bA^d)/G \ar[l]_-{\phi_U = \beta/G}
        }
        \]
        where lower script $U$ denotes restriction to $U$ and $\beta$ is a \emph{weighted blowup} of $\bA^d$ along $H$ with weight $w$. N.B. the group $G=\mathfrak{D}(\Cl(U)) \subset \bT$ is toric and so $\beta/G$ is well-defined.
\end{proposition}
\begin{proof}
    We use the fact that $\Sigma_X$ is given by $\Sigma^*_S(u_{k+1})$---the so-called \emph{star subdivision} of $\Sigma_S$ at $u_{k+1}$ \cite[\S11.1]{CoxLittleSchenckToricVarieties}. It is instructive to recall how this works. First, note that $\Sigma_X$ and $\Sigma_S$ are fans in the common space $N_{\bR} \simeq \bR^d$. The point here is that the center $C\subset S$ has codimension $k$ (see \autoref{thm.MoriTheoryOfFanoVarieties} (a)) and it corresponds by the orbit-cone correspondence to the $k$-dimensional cone $\gamma \coloneqq \langle u_1,\dots, u_k \rangle_{\R\geq 0} \in \Sigma_S$. In particular, 
    \[
    \{\sigma \in \Sigma_S \mid \sigma \not\supset \gamma\} = \Sigma_{S\setminus C} =  \Sigma_{X \setminus P_{k+1}} = \{\sigma \in \Sigma_X \mid \sigma \notin u_{k+1} \}.
    \]
    Thus, for $\sigma \in \Sigma_S$, the conditions $\gamma \subset \sigma$ and $u_{k+1} \in \sigma$ are equivalent. So $\Sigma_{S \setminus C} = \{\sigma \in \Sigma_S \mid u_{k+1} \notin \sigma\}$. Let its complement be $\Pi_C=\{\sigma \in \Sigma_S \mid \sigma \ni u_{k+1}\} = \{\sigma \in \Sigma_S \mid \sigma \supset \gamma\}$.
    
    With the above in place, we have
    \[
    \Sigma_X= \Sigma^*_S(u_{k+1}) \coloneqq \Pi_{P_{k+1}} \sqcup \Sigma_{S \setminus C}.
    \]
    where
    \[
    \Pi_{P_{k+1}} \coloneqq \bigsqcup_{\sigma \in \Pi_C}  \{ \langle \tau, u_{k+1} \rangle_{\bR_{\geq 0}} \mid  \tau \subset \sigma, \tau \in \Sigma_{S\setminus C} \} =  \{\sigma \in \Sigma_X \mid \sigma \supset \rho_{k+1}\}.
    \]
    This construction readily implies that $\Sigma_S$ is simplicial if and only if so is $\Sigma_X$. That is, $S$ is $\bQ$-factorial if and only if so is $X$. 

    Let $y \in S^{\bT}$ and $y \in U \subset S$ be the corresponding standard purely $\bQ$-factorial affine toric chart; see \autoref{subsec.QFactorialToricSingularities}. If $y \notin C$ there is nothing to do, so we may assume that $y\in C$. In particular, we may say that $P_1,\ldots,P_k,P_{k+2},\ldots, P_{d+1}$ are the $\bT$-invariant prime divisors on $U$. Moreover, $C_U = C \cap U$ is the scheme-theoretic intersection of the $P_1,\ldots,P_k$ in $U$. This explains the left-hand side part of the displayed diagram. Let us move to the right-hand side.
    
    We may take the weighted blowup of $\bA^d$ along $H$ with weights $w\in \bN^k$. This is defined as the $\Proj$ of the \emph{weighted Rees $A$-algebra} $\mathcal{R} \coloneqq \bigoplus_{n\in \bN} (t_1,\ldots,t_k)^n t^n \subset A[t]$, where $A\coloneqq \kay[t_1,\ldots,t_n]$ and one declares $t_i$ to have degree $w_i$. In particular, the pullback of $\beta\:\Bl_H^{\mathrm{w}} \bA^d \to \bA^d$ along $H \subset \bA^d$ is a weighted projective space $\beta_H \:\bP_H(w_1,\ldots,w_k) \to H=\bA^{d-k}$.

    However, $\beta\: \Bl_H^{\mathrm{w}} \bA^d \to \bA^d$ admits a purely toric description. Namely, 
    \[
    \Sigma_{\Bl_H^{\mathrm{w}}} = \Sigma^*_{\bA^d}(w_1 e_1 + \cdots + w_k e_k)
    \]  
    where $\Sigma_{\bA^d}$ is the standard fan of $\bA^d$; see \autoref{subsec.QFactorialToricSingularities}. In particular, we may further act by any toric subgroup $G\subset \bT$ defined by a matrix $\Lambda = (\lambda_1 \cdots \lambda_d)_{d \times d} \in \End_{\bZ}(\bZ^d)$ to obtain
    \[
    \Sigma_{\Bl_H^{\mathrm{w}}/G} = \Sigma^*_{\bA^d/G}(w_1 \lambda_1 + \cdots + w_k \lambda_k ). 
    \]
    This can be made to coincide with the fan of $\Sigma_{X_U} = \Sigma^*_U(u_{k+1})$ by taking $\Lambda = \Upsilon$, i.e., $\lambda_i=u_i$. Indeed, $u_{k+1} = b'_1 u_1 +\cdots +b'_k u_k \in \langle w_1 u_1 +\cdots +w_k u_k \rangle_{\bQ_{>0}}$ where $b'_i \coloneqq b_i/b_{k+1}$. 
\end{proof}
   
\begin{proposition}\label{prop.CharacterizationInert}
    With notation as in \autoref{prop.LocalStructureExtremalDivContraction}, 
    the following are equivalent:
    \begin{enumerate}
        \item $u_{k+1} \in \langle u_1,\ldots,u_k \rangle_{\bN}$
        \item $b_{k+1} = 1$.
        \item $X_U$ admits an open covering by purely $\bQ$-factorial toric affine charts $\bA^d \to \bA^d/H=V$ such that $G \subset H $ are toric subgroups of $\bT$.
    \end{enumerate}
\end{proposition}
\begin{proof}
     We use the notation in the proof of \autoref{prop.LocalStructureExtremalDivContraction}; we resume where we left off there. Recall that $y$ corresponds to \[
     \sigma \coloneqq  \langle \gamma, u_{k+2}, \ldots, u_{d+1} \rangle_{\bR_{\geq 0}} = \langle u_1,\ldots,u_k, u_{k+2}, \ldots, u_{d+1} \rangle_{\bR_{\geq 0}} \in \Sigma_S(d).
     \]
     Observe that there are $k$ points in $X^{\bT}$ in the fiber of $y$. These correspond to the cones in $\Sigma_X(d)$ given by 
     \[
     \sigma_i \coloneqq \langle u_1, \ldots, u_{i-1},u_{k+1}, u_{i+1},\ldots,u_k, u_{k+2},\ldots,u_{d+1} \rangle_{\bR_{\geq 0}},
     \]
     where we just took the generators of $\sigma$ and replaced $u_i$ by $u_{k+1}$. Thus, $X_U$ can be covered by the standard open sets $V_1,\ldots,V_k$ corresponding to these points. For ease of notation and without loss of generality, let us focus on $V \coloneqq V_k$, which corresponds to $\sigma_k = \langle u_1,\ldots,u_{k-1},u_{k+1}, u_{k+2}, \ldots, u_{d+1} \rangle_{\bR_{\geq 0}}$.

     With the above in place, recall that
     \[
     \Cl(U) = N/\langle u_1,\dots,u_k, u_{k+2},\ldots, u_{d+1} \rangle_{\bZ}
     \]
     whereas 
     \[
     \Cl(V) = N/\langle u_1,\dots,u_{k-1}, u_{k+1},\ldots, u_{d+1} \rangle_{\bZ}
     \]
     Then, we observe that
     \begin{align*}
     &G=D(\Cl(U)) \subset H=D(\Cl(V))\\
     \Longleftrightarrow{} & G/G\cap H = 0 \\
     \Longleftrightarrow{} &\langle u_1,\dots,u_k, u_{k+1}, u_{k+2},\ldots, u_{d+1} \rangle_{\bZ}/ \langle u_1,\dots,u_k, u_{k+2},\ldots, u_{d+1} \rangle_{\bZ} = 0 \\
     \Longleftrightarrow{} & u_{k+1} \in \langle u_1,\ldots,u_k, u_{k+2},\ldots,u_{d+1} \rangle_{\bZ}
     \end{align*}
Since $X$ is $\bQ$-factorial, the vectors $u_1,\ldots,u_k,u_{k+2},\ldots,u_{d+1}$ form a basis of $N_{\bR}$ (for they generate $\sigma$). Then, this is further equivalent to $ (b_1/b_{k+1},\ldots,b_k/b_{k+1}) \in \bN^k$ and so to $b_{k+1}=1$ as we chose these coefficients having no common factors.
\end{proof}

\begin{definition}[Inert extremal divisorial contractions]
    With notation as in \autoref{prop.LocalStructureExtremalDivContraction}, we say that $\phi$ is \emph{inert} if any of the equivalent conditions in \autoref{prop.CharacterizationInert} hold.
\end{definition}

\subsection{Some convex geometry}
We collect here a couple of lemmas on convex geometry that we will use throughout the proofs of our main results.

\begin{lemma}\label{lemma.FindingPlatticePoints}
    Set $v_1,\dots,v_m \in \Z^n \setminus \{0\}$ and let $C \coloneqq \langle v_1,\ldots,v_m \rangle_{\bR_{\geq 0}} \subset \bR^n$ be the corresponding rational cone. Suppose that $m>n$ and that $C$ is strongly convex. Then,
    \[
        \langle v_1,\ldots,v_m \rangle_{[0,1)} \cap \bZ^n \neq 0.
    \]
    In fact, given a non-trivial relation $c_1 v_1 + \cdots + c_m v_m =0$ with $c_1,\ldots,c_m \in \bZ$ and $c_{i_0} \neq 0$, we have that
    \[
    0 \neq \sum_{i \: c_i c_{i_0} > 0} v_i \in  \langle v_1,\ldots,v_m \rangle_{[0,1)} \cap \bZ^n.
    \]
\end{lemma}
\begin{proof}
    Since $m > n$, there is a relation $ \sum_{i=1}^m c_iv_i = 0$ with $c_i \in \Z$ (with not all $c_i$ equal to zero). Since $C$ is strongly convex, at least one of the coefficients is negative and at least one of them is positive. Rearrange the relation as
 \[
        a_1v_1 + \dots + a_kv_k = b_{k+l}v_{k+l} + \dots + b_{m}v_m,
    \]
where $0<a_{1}\leq \cdots \leq a_k \in \bN$ and $0<b_{k+l}\leq \cdots\leq b_{m} \in \bN$. Once again, $k \geq 1$ and $m \geq k+l$ as $C$ is strongly convex. We may assume, without loss of generality, that $b_m \geq a_k$. We do two cases depending on whether or not the inequality is strict.

If $b_m > a_k$, then add $\sum_{j=k+l}^{m-1} (b_m-b_j) v_j$ on both sides to obtain that
\[
 a_1v_1 + \dots + a_kv_k + (b_m-b_{k+l})v_{k+l} + \dots + (b_{m}-b_{m-1})v_{m-1} = b_m(v_{k+l} + \cdots + v_m)
\]
Dividing by $b_m$ yields that $v_{k+l} + \cdots + v_m$ is a nonzero lattice point in $ \langle v_1,\ldots,v_m \rangle_{[0,1)}$ (where we use strong convexity to say that this point is not zero).

If $b_m = a_k$, set $c\coloneqq b_m+1$ and add $\sum_{j=k+l}^{m-1} (c-b_j) v_j + v_m$ to both sides to get the relation
    \[
        a_1v_1 + \dots + a_kv_k + (c - b_{k+1})v_{k+l} + \dots + (c - b_{m-1})v_{m-1} + v_m = c(v_{k+l} + \dots + v_m)
    \]
   Dividing by $c$ yields that $v_{k+l} + \cdots + v_m$ is the required point, as before.
\end{proof}

\begin{lemma}\label{lemma.LatticePointsArepLatticePoints}
Set $v_1,\dots,v_m \in \Z^n\setminus 0$ and let $C \coloneqq \langle v_1,\ldots,v_m \rangle_{\bR_{\geq 0}} \subset \bR^n$ be the corresponding rational cone. Suppose that 
\begin{equation}\label{eqMagicClaim}
        C \cap \Z^n = \langle v_1,\ldots,v_m \rangle_{\bN}.
    \end{equation}
Then, the inclusion
\[
 \langle v_1,\ldots,v_m \rangle_{[0,1)\cap \bZ[1/p]} \cap \bZ^n \subset  \langle v_1,\ldots,v_m \rangle_{[0,1)} \cap \bZ^n
\]
is an equality.
\end{lemma}
\begin{proof}
   Let $v$ be a lattice point in $\langle v_1,\ldots,v_m \rangle_{[0,1)}$. Since the generators $v_i$ are in $\bQ^n$,
    \[
    \langle v_1,\ldots,v_m \rangle_{[0,1)} \cap  \bQ^n = \langle v_1,\ldots,v_m \rangle_{[0,1) \cap \bQ}.
    \]
    In particular, we may write $v=c_1v_1 + \cdots +c_m v_m$ with $c_i \in [0,1) \cap \bQ$. We may assume that $c_i \in [0,1) \cap \bZ \cdot r^{-1}$ for some $r \in \bN$ sufficiently divisible. The hypothesis \autoref{eqMagicClaim} lets us write
    \[
    c_1v_1 + \cdots +c_m v_m = n_1v_1 + \cdots +n_m v_m
    \]
    for some $n_1,\ldots,n_m \in \bN$. Write a partition
    \[
    [m] \coloneqq \{1,\ldots,m\} = I \sqcup J \sqcup K
    \]
    where $J=\{j \mid c_j, n_j \neq 0\}$, $I=\{i \mid c_i \neq 0, n_i = 0\}$, and $K=\{k \mid c_k = 0, n_k \neq 0\}$. Then we may write
    \[
    \sum_{i \in I} c_i v_i = \sum_{j\in J} (n_j-c_j) v_j + \sum_{k \in K} n_k v_k.
    \]
Now, write $c_i=a_i/r$ (so $0\leq a_i<r$) and $\tilde{n}_j \coloneqq n_j - c_j>0$. Then
\[
\sum_{i \in I} a_i v_i = \sum_{j\in J} r \tilde{n}_j v_j + \sum_{k \in K} r n_k v_k.
\]
 Letting $e\gg 0$ such that $q = p^e\geq r$, add $\sum_{j\in J}(qn_j-r\tilde{n}_j)v_j + \sum_{k \in K}(q-r)n_k v_k$ on both sides to get
 \[
 \sum_{i \in I} a_i v_i + \sum_{j\in J}(qn_j-r\tilde{n}_j)v_j + \sum_{k \in K}(q-r)n_k v_k = q (n_1v_1 + \cdots +n_m v_m) =q v.
 \]
 In particular,
 \[
 v= \sum_{i \in I} \frac{a_i}{q} v_i + \sum_{j\in J}\frac{(q-r)n_j+a_j}{q}v_j + \sum_{k \in K}\frac{(q-r)n_k}{q} v_k.
 \]
 We are done if we can arrange $r$ and $q$ such that the displayed coefficients (which are all non-negative and have integral numerators) are strictly less than $1$. Those of $I$ are fine as $q\geq r>a_i$. For those in $J$ or $K$, we need to arrange for
 \[
 \frac{q}{r} < 1 + \frac{1-c_j}{n_j-1} \quad \forall j \in J \qquad \text{and} \qquad
     \frac{q}{r} < 1 + \frac{1}{n_k - 1} \quad \forall k \in K.
    \]
    Let $\bN \ni l \gg 0$ be such that
    \[
    1/l <  \min \left\{ \left\{\frac{1}{n_k - 1}\right\}_{k\in K} \cup \left\{\frac{1-c_j}{n_j-1}\right\}_{j \in J} \right\}.
    \]
    It suffices to find $e \gg 0$ and $r$ sufficiently divisible such that $1\leq q/r < 1 + 1/l$. First, choose $e \gg 0$ such that $ t \coloneqq \left\lceil q/r \right\rceil > l+1$. Then, for $s \coloneqq (t-1)r$ we readily verify that  $1 \leq q/s < 1 + 1/l$; as required.
\end{proof}

\section{Basic Frobenius Geometry of Toric Varieties} In this section, we continue using the notation of \autoref{sec.ToricGeometryPreliminaries} but we relax the condition on $\bQ$-factoriality back to normality. In particular, we set $X$ as a (normal and projective) $d$-dimensional toric variety.

According to Achinger \cite{AchingerCharacterizationOfToricVarieties}, smooth toric varieties are characterized (among smooth projective varieties) as those such that the Frobenius pushforwards of invertible sheaves split as direct sums of invertible sheaves. Moreover, we have the following formula:

\begin{theorem}[{\cite[Theorem 2]{AchingerCharacterizationOfToricVarieties}, \cf \cite[Theorem 1]{ThomsenFrobeniusDirectImagesOfLineBundlesToricVarieties}, \cite{BogvadSplittingFrobenius}}]\label{thm.SplittingFormula}
Let $D$ be a divisor on $X$ and $0 \neq e \in \bN$. Then, 
\[
    F_*^e\sO_X(-D) \simeq \bigoplus_{[E] \in \Cl(X) }\sO_X(-E)^{\oplus m_D(E;q)},
\]
where $m_D(E;q)$ is the number of $\bT$-invariant divisors with coefficients in $\{0,\dots,q-1\}$ linearly equivalent to $qE-D$, i.e., the number of $\bT$-invariant divisors in the linear system $|qE-D|$ with coefficients $<q$. 
\end{theorem}

This lets us compute the Frobenius-trace kernel of $X$ right away.

\begin{corollary}\label{cor.EforTorics}
   The Frobenius-trace kernel of $X$ is given by
    \[
\sE_{X,e} \simeq \bigoplus_{0 \neq [E] \in \Cl(X)} \sO_X(E)^{\oplus m(E;q)},
\]
where $m(E;q)\coloneqq m_0(E;q)$ is the number of $\bT$-invariant divisors on $X$ with coefficients in $\{0,\ldots,q-1\}$ that are linearly equivalent to $q E$. In particular, $\sE_{X,e}$ is pseudo-effective.
\end{corollary}

\begin{question}
    For which varieties of Fano-type $X$ is $\sE_{X,e}$ pseudo-effective for all $e>0$? Is this true for globally $F$-regular varieties?
\end{question}

Note that if $E$ is a divisor on $X$ (or rather a divisor class) such that $m(E;q) \neq 0$ for some $e>0$, then $m(E;q') \neq 0$ for all $q' \geq q$. This justifies the following definition.

\begin{definition}[Frobenius support]\label{def.SupportingDivisors}
    The following definitions are in order:
    \begin{enumerate}
        \item  We say that a divisor $E$ on $X$, or rather its divisor class in $\Cl(X)$, \emph{supports the Frobenius-trace kernel} if $m(E;q) \neq 0$ for some $e>0$.
        \item  We define the \emph{Frobenius support} of $X$ as the set of numerical classes $\FSupp(X) \subset N^1(X)$ of divisor classes supporting the Frobenius-trace kernel.
        \item Given $\llbracket E \rrbracket \in \FSupp(X)$, we define
    \[
    \tilde{m}(E;q) = \sum_{\{[D] \in \Cl(X) \mid D \equiv E\}} m(D;q) =  \sum_{[T]  \in \Cl(X)_{\mathrm{tor}}} m(E+T;q).
    \]
    \item We also introduce the sets
    \begin{align*}
      \BFS(X)  &\coloneqq \bigc(X) \cap \FSupp(X), \\
       \AFS(X) & \coloneqq \ample(X) \cap \FSupp(X), \text{ and} \\
       \NFS(X) & \coloneqq \nef(X) \cap \FSupp(X),
    \end{align*}
    which we call, respectively, the \emph{big/ample/nef Frobenius supports} of $X$.
    \end{enumerate}
    
\end{definition}
\begin{remark}
    There are finitely many divisor classes that support the Frobenius-trace kernel, and so $\FSupp(X)$ is a finite set.  If $X$ is smooth, then $\FSupp(X)$ is the set of divisor classes supporting the Frobenius-trace kernel. More generally, the Kunz theorem can be rephrased by saying that $X$ is smooth if and only if every divisor supporting its Frobenius trace kernel is Cartier.
\end{remark}
We may reinterpret \autoref{thm.SplittingFormula} to compute $\FSupp(X)$ as follows. Recall that we may think of $N^1(X) \simeq \bZ^{\rho}$ as the free part of $\Cl(X)$ and as a lattice inside $N^1(X)_{\bR} \simeq \bR^{\rho}$. Let us write
\[
\Cl(X)_{\mathrm{tor}} = \Cl(X)_{p\textrm{-tor}} \oplus \Cl(X)_{\textrm{tor}}'
\]
where $\Cl(X)_{p\textrm{-tor}}$ is the largest $p$-subgroup of $\Cl(X)_{\textrm{tor}}$. In particular, $\cdot p \:  \Cl(X)_{\textrm{tor}}' \to  \Cl(X)_{\textrm{tor}}'$ is an automorphism. Let $e_0>0$ be minimal such that $q_0 \Cl(X)_{p\textrm{-tor}} = 0$.

We further consider the following ``half-open'' convex polytope
\begin{equation*}
    \sQ \coloneqq \sQ_X \coloneqq \langle \pi_1,\ldots,\pi_r \rangle_{[0,1)} \subset \eff(X) \subset N^1(X)_{\bR}.
\end{equation*}

\begin{corollary} \label{cor.SupportOfE_X}
    With notation as above,
     \[
        \FSupp(X) = \sQ_X \cap N^1(X)\setminus 0.
    \]
    In particular, no torsion divisor class supports the Frobenius-trace kernel and
    \[
        \BFS(X) = \sQ_X^{\circ} \cap N^1(X).
    \]
\end{corollary}
\begin{proof}
The formula for the big Frobenius support is a direct consequence due to \autoref{cor.TheBigCone}. It follows directly from \autoref{cor.EforTorics} that
\[
\FSupp(X) \subset \sQ_{\bZ[1/p]} \cap N^1(X)  \subset \sQ \cap N^1(X)
\]
where
\[
\sQ_{\bZ[1/p]} \coloneqq \langle \pi_1,\ldots,\pi_r\rangle_{[0,1) \cap \bZ[1/p]}. 
\]
Moreover, $0 \notin \FSupp(X)$ as a torsion divisor class cannot support the Frobenius trace kernel by \autoref{prop.PseudoEffectiveCone}. In conclusion,
 \[
        \FSupp(X) \subset \sQ_X \cap N^1(X)\setminus 0.
    \]
    
    For the converse containment, observe that \autoref{lemma.LatticePointsArepLatticePoints} and \autoref{prop.PseudoEffectiveCone} imply that
\[
\sQ_{\bZ[1/p]} \cap N^1(X) \supset \sQ \cap  N^1(X),
\] 
so that it remains to show the inclusion
\[
\FSupp(X) \supset \sQ_{\bZ[1/p]} \cap N^1(X).
\]
Let $E$ be a divisor whose numerical class sits on the right-hand side of this containment. This means that there is $e \gg 0$ such that
\[
q E \equiv c_1 P_1 + \cdots +c_r P_r
\]
for some integers $0 \leq c_1,\ldots,c_r \leq q-1$. Therefore,
\[
q E + T + T' \sim c_1 P_1 + \cdots +c_r P_r
\]
for some uniquely determined divisor classes $[T] \in \Cl(X)_{p\textrm{-tor}}$ and $[T']\in \Cl(X)_{\textrm{tor}}'$. Write $T'= qT''$ for some uniquely determined $[T'']\in \Cl(X)_{\textrm{tor}}'$. Multiplying the displayed linear equivalence by $q_0$ then yields
\[
qq_0(E +T'') \sim c'_1 P_1 + \cdots +c'_r P_r
\]
where $c'_i \coloneqq q_0 c_i \in [0,q_0(q-1)] \subset [0,qq_0-1]$. In other words, $\sO_X(E+T'')$ is a direct summand of $\sE_{X,e+e_0}$ and so $E'\coloneqq E+T''$ supports the Frobenius trace kernel. Since $E' \equiv E$, this shows that $\llbracket E \rrbracket \in \FSupp(X)$ and so the required, remaining inclusion.
\end{proof}

\begin{corollary}\label{cor.multiplicityOfDirectSummands}
Let $E$ be a divisor supporting the Frobenius-trace kernel of $X$. Then, there is $\alpha(E) \in \bQ \cap [0,1]$ such that
    \[
    \tilde{m}(E;q) = \alpha(E) q^{d}+ O(q^{d-1}).
    \]
    Moreover, $\alpha(E)>0$ if and only if $E$ is big. In fact,
    \[
    \vol_X(E) \geq \frac{\alpha(E) d!}{|\Cl(X)_{\mathrm{tor}}|}.
    \]
    In particular, $\BFS(X) \neq \emptyset$.
\end{corollary}
\begin{proof}
   Let $\bm{c} \in [0,1)^{\times r} $ be such that $
   \llbracket E \rrbracket=  c_1 \pi_1 + \cdots +c_r \pi_r \in N^1(X).$ 
    According to the proof of \autoref{cor.SupportOfE_X}, $\tilde{m}(E;q)$ is the number of any such $\bm{c}$'s with entries in $\bZ \cdot q^{-1}$ as long as $e \geq e_0$. That is, if $e>e_0$ then $\tilde{m}(E;q)$ is the cardinality of the set
    \[
    (\bm{c}+K) \cap \big([0,1) \cap \bZ \cdot q^{-1} \big)^{\times r} \subset \bR^r
    \]
    where $K \simeq \bR^{d}$ is the kernel of the $\rho \times r$ real matrix whose $i$-th column is the vector $\pi_i$ in some fixed basis of $N^1(X)_{\bR}$. In particular, $\alpha(E)$ is the measure of the polytope
    \[
    \sP \coloneqq  K \cap (-\bm{c}+[0,1]^{\times r}) \subset K \simeq \bR^d.
    \]
    This proves the asymptotic equality regarding $\tilde{m}(E;q)$. However, it also tells us that $\alpha(E)>0$ if and only if $\sP \subset \bR^d$ is a $d$-dimensional polytope.

    On the other hand, by \autoref{cor.TheBigCone}, we conclude that $E$ is big if and only if $(\bm{c}+K)$ intersects $(0,1)^{\times r}$. That is, $E$ is big if and only if
    \[
    \emptyset \neq K \cap (-\bm{c}+(0,1)^{\times r})  \subset  \sP^{\circ}.
    \]
    In which case $\dim \sP = d$ and so $\alpha(E)>0$. Hence, $\alpha(E)$ is positive if $E$ is big.
    
    For the converse, Recall that a divisor is big if and only if it has positive volume. It then suffices to show that $\vol_X(E)$ is bounded below by $\alpha(E)d!/|\Cl(X)_{\mathrm{tor}}|$. Observe that
    \[
    \vol_X(E)/d! \coloneqq \limsup_{n \rightarrow \infty} \frac{h^0(X,nE)}{n^d} \geq \limsup_{e \rightarrow \infty} \frac{h^0(X,qE)}{q^d} 
    \]
    However, twisting 
    \[
    F^e_*\sO_X = \bigoplus_{[D] \in \Cl(X)} \sO_X(-D)^{\oplus m(D;q)}
    \] by $\sO_X(E)$ and using the projection formula yields
    \[
    F^e_* \sO_X(qE) =\bigoplus_{[D] \in \Cl(X)} \sO_X(E-D)^{\oplus m(D;q)}.
    \]
   Taking global sections (and using that $F^e$ is affine) then yields
    \[
    h^0(X,qE)   = \sum_{ [D] \in \Cl(X) } m(D;q)h^0(X,E-D) \geq m(E;q).
    \]
Therefore,
\[
\sum_{[T] \in \Cl(X)_{\mathrm{tor}}} h^0(X,q(E+T)) \geq \tilde{m}(E;q)
\]
    Dividing by $q^d$ and taking $e \rightarrow \infty$, we conclude that 
    \[
   \sum_{[T] \in \Cl(X)_{\mathrm{tor}}} \vol_X(E+T) \geq \alpha(E) d!.
    \]
    However, being volumes a numerical invariant, the left-hand side of this inequality is $|\Cl(X)_{\mathrm{tor}}| \vol_X(E)$; as required.

For the final statement, observe that
\[
q^d-1 = \sum_{[E] \in \Cl(X)} m(E;q) =  \sum_{\llbracket E \rrbracket \in \FSupp(X)} \tilde{m}(E;q
)
\]
and so
\[
1 = \sum_{\llbracket E \rrbracket \in \FSupp(X)} \alpha(E) =  \sum_{\llbracket E \rrbracket \in \BFS(X)} \alpha(E).
\]
Since the latter is a sum of positive numbers, $\BFS(X)$ cannot be empty.
\end{proof}

From the argument in \autoref{cor.multiplicityOfDirectSummands}, the following exact formula is obtained for the volume of divisors on smooth projective varieties. It illustrates the principle behind this work, namely basic classical invariants are determined by the Frobenius.

\begin{scholium}[Volume formula]
    Let $X$ be a smooth $d$-dimensional toric variety. The following formula holds for every divisor $E$ on $X$:
    \[
    \vol_X(E)/d! = \sum_{[D] \in \BFS(X)} \alpha(D) h^0(X,E-D).
    \]
    In particular, $E$ is big if and only if $h^0(X,E-D) \neq 0$ for some $[D] \in \BFS(X)$. That is, for every big divisor $E$ on $X$ there is $[D] \in \BFS(X)$ and $N \geq 0 $ such that $E \sim D + N$.
\end{scholium}

\begin{corollary} \label{cor.FrobeniusSupport}
    With notation as in \autoref{rem.NefConesOfQFactorialModifications}, the toric small $\bQ$-factorial modifications of $X$ share the same Frobenius support.
\end{corollary}

\begin{remark}[Frobenius support and finite toric quotients] \label{rem.FrobeniusSupportAndQuotients}
    Let $X$ be a $\bQ$-factorial toric variety and $G \subset \bT$ be a toric subgroup. Recall that $G$ is determined by a square $\bZ$-matrix $\Lambda$ of size $d = \dim X$ with a non-zero determinant; see \autoref{subsec.QFactorialToricSingularities}. Moreover, the quotient $f\: X \to X/G$ is given by the fan $\Sigma_{X/G} = \{\Lambda(\sigma) \mid \sigma \in \Sigma_X\}$. Consider the pullback homomorphism $f^*\: \Cl(X/G) \to \Cl(X)$, which is well-defined as $f$ is a finite cover; see, e.g., \cite[\S2.2]{SchwedeTuckerTestIdealFiniteMaps}. Observe that this is a quotient map. Indeed, the prime $\bT$-invariant divisors are preserved under this pullback, i.e., ``$f^* P_i = P_i$'' for all $i=1,\ldots,r$. What changes are the relations, which do change under the action of $\Lambda^{\top}$. This is formally explained by the following commutative diagram of short exact sequences
    \[
    \xymatrixcolsep{4.5pc}\xymatrix{
    0 \ar[r]   & \bZ^d \ar[d]_-{\Lambda^{\top}}\ar[r]^-{[\Lambda u_1 \cdots \Lambda u_r]^{\top}} & \bZ^r \ar[d]^-{\id} \ar[r] & \Cl(X/G) \ar[r]\ar[d]^-{f^*}\ar[r] & 0 \\
 0  \ar[r] & \bZ^d \ar[r]^-{[u_1\cdots u_r]^{\top}} & \bZ^r  \ar[r] & \Cl(X) \ar[r]  & 0 
    }
    \]

This further implies that $\ker f^* \xrightarrow{\simeq} \coker \Lambda^{\top}$ by the snake lemma. In particular, $\ker f^*$ is a finite group of order $|\det \Lambda^{\top}| = |G|$ and so
\[
(f^*)^{-1}(\Cl(X)_{\mathrm{tor}}) = \Cl(X/G)_{\mathrm{tor}}.
\]
Consequently, $f^*\: \Cl(X/G) \to \Cl(X)$ defines a $\bZ$-linear isomorphism 
\[
f^*\: N^1(X/G)_{\bR} \xrightarrow{\simeq} N^1(X)_{\bR}
\] and so an $\bR$-linear one $f^* \: N^1(X/G)_{\bR} \to N^1(X)_{\bR}$. Moreover, under this isomorphism, the corresponding pseudo-effective (resp. big/nef/ample) cones are identified. More importantly, $f^*$ establishes a bijection 
    \[
    f^*\:\FSupp(X/G) \xrightarrow{\simeq} \FSupp(X)
    \]
    between the Frobenius supports, as well as between the corresponding big/nef/ample Frobenius supports; respectively. Also, we readily see that  $\alpha(f^* E ) = \alpha(E)$  for all $\llbracket E \rrbracket \in \FSupp(X/G)$ using the description of $\alpha(-)$ as the measure of a certain polytope in $\bR^d$ as explained in the proof of \autoref{cor.multiplicityOfDirectSummands}. 
\end{remark}

\begin{remark}[Asymptotic behavior] \label{rem.AsymptoticBehavior}
If $X$ is smooth, there is a projective system \[ 
\sE_{X,1} \twoheadleftarrow  \sE_{X,2} \twoheadleftarrow  \sE_{X,3} \twoheadleftarrow \cdots .\]
More generally, there are exact sequences
\[
 0 \to F^{e}_* \left(\sE_{X,e'} 
\otimes \omega_X^{1-q }\right) \to \sE_{X,e+e'} \to \sE_{X,e} \to 0
\]
for all $e,e'>0$. Since $X$ is $F$-split, these exact sequences are split and so
\[
\sE_{X,e+e'} = \sE_{X,e} \oplus F^{e}_* \left(\sE_{X,e'} 
\otimes \omega_X^{1-q }\right).
\]
See \cite[Remark 5.2]{CarvajalRojasPatakfalviVarietieswithAmpleFrobenius-traceKernel} for details. Therefore, when we say that $\sE_{X,e}$ has a positivity property such as nefness or ampleness for all $e>0$, we mean that this property holds for 
\[
\sF_{e,e'} \coloneqq F^e_*(\sE_{X,e'} \otimes \omega_X^{1-q}) = (F^e_* \sB_{X,e'}^1)^{\vee} 
\] 
for all $e \geq 0$ and any $e'>0$. Note that 
\[
\sF_{0,e'} = \sE_{X,e'}, \quad \text{and} \quad \sE_{X,e+e'}=\sF_{0,e'} \oplus \sF_{1,e'} \oplus \cdots \oplus \sF_{e,e'}.
\] 
Also, $\rk \sF_{e,e'} = q^d(q'^d-1)$.

We then conclude the following. For $[E] \in \FSupp(X)$, we have
\[
F^e_*\sO_X(-E) = \bigoplus_{[D] \in \FSupp(X)} \sO_{X}(-D)^{\oplus m_E(D;q)}.
\]
That is, every divisor $D$ that appears in the above direct sum must be in the Frobenius support of $X$. In other words, $m_E(D;q)=0$ if $[D] \notin \FSupp(X)$. We may rewrite the above as
\[
\sE_{X,e}[E] \coloneqq F^e_*\sO_X(E+(1-q)K_X) = (F^e_*\sO_X(-E))^{\vee} = \bigoplus_{[D] \in \FSupp(X)} \sO_{X}(D)^{\oplus m_E(D;q)}.
\]
Therefore,
\[
\sF_{e,e'} = \bigoplus_{[E] \in \FSupp(X)} \sE_{X,e}[E]^{\oplus m(E;q')}  = \bigoplus_{[D] \in \FSupp(X)} \sO_{X}(D)^{\oplus m(D;q,q')},
\]
where
\[
m(D;q,q') \coloneqq \sum_{[E] \in \FSupp(X)} m(E;q') m_E(D;q).
\]
We will come back to this in \autoref{sec.AmpleFSignHomogeneity}.
\end{remark}

\subsubsection{Relationship with boundaries and log structures} \label{subsubse.LogStructures}
Recall that a \emph{log pair} $(X,\Delta)$ is the combined data of a (normal quasi-projective) variety $X$ and a \emph{boundary} $\Delta$ on $X$. This means that $\Delta$ is an effective $\bQ$-divisor on $X$ with coefficients $\leq 1$ such that $K_X + \Delta$ is $\bQ$-Cartier. That is, there is $0 \neq n \in \bN$ such that $n(K_X+ \Delta)$ is an integral Cartier divisor. The smallest such $n$ is called \emph{(Cartier) index} of the pair $(X,\Delta)$. 

It is worth noting that the $\bQ$-Cartier condition above is free if $X$ is $\bQ$-factorial. In that case, what multiplying by $n$ does is simply to clear the denominators in the coefficients of $K_X+\Delta$ and then be divisible enough to annihilate the corresponding integral divisor in the torsion group $\Cl(X)/\Pic(X)$. That is, the index of $(X,\Delta)$ is the index of $K_X + \Delta$ as an element in the group $\DIV(X)_{\bQ}/\DIV(X)$, say $0 \neq a \in \bN$, times the index of $a(K_X+\Delta)$ in $\Cl(X)/\Pic(X)$. 

We will be, however, interested in another kind of index for log pairs. Namely, the index of $K_X+\Delta$ as an element of $\Cl(X)_{\bQ}=N^1(X)_{\bQ}$ modulo $N^1(X)$, i.e., the smallest $0\neq n \in \bN$ such that $n(K_X+\Delta) \sim_{\bQ} E$ for some integral divisor $E$ on $X$.\footnote{More succinctly, we do not care about clearing denominators.} We will refer to it as the \emph{class index} of $(X,\Delta)$. We are unaware of any other term used in the literature for it. Of course, the class index always divides the Cartier index.  

Let $X$ be a $\bQ$-factorial $d$-dimensional toric variety, in particular $-K_X=P_1 + \cdots +P_r$. Then a \emph{toric boundary} $\Delta$ is defined as a $\bQ$-divisor $\Delta = \delta_1 P_1 + \cdots +\delta_r P_r$ with $\delta_1,\ldots,\delta_r \in (0,1]$. As mentioned above, the effective $\bQ$-divisor
\[
-(K_X + \Delta) = (1-\delta_1)P_1 +\cdots +(1-\delta_r) P_r \eqqcolon \Delta' 
\]
is automatically $\bQ$-Cartier. We refer to the pair $(X,\Delta)$ as a \emph{toric log pair}.

We are interested in toric log pairs $(X,\Delta)$ of class index $1$. To see why, let $(X,\Delta)$ be one and write $\Delta' \sim_{\bQ} E$ for $E$ some integral divisor. Then $\Delta$ is a big divisor and the associated integral divisor $E$ must be in $\FSupp(X)$ as long as $\Delta \neq -K_X$. Conversely, every element in $\FSupp(X)$ arises in this way. Moreover, if $\lfloor \Delta \rfloor = 0$ then the associated divisor $E$ is big and all the big divisors in $\FSupp(X)$ are obtained in this way. Of course, many different toric boundaries $\Delta$ may correspond to a divisor $E$ in $\FSupp(X)$. However, as explained in the proof of \autoref{cor.multiplicityOfDirectSummands}, the set of such boundaries is parametrized by the rational points of a polytope in $\bR^d$, whose measure is $\alpha(E)$.

The following is a well-known fact among experts. We prove it for the sake of completeness.

\begin{proposition} \label{pro.KLTnessToriclOGpAIRS}
    Let $(X,\Delta)$ be a toric log pair where $X$ is a $d$-dimensional $\bQ$-factorial toric variety. The following statements are equivalent:
    \begin{enumerate}
        \item $(X,\Delta)$ is $F$-regular.
        \item $(X,\Delta)$ is KLT (i.e. Kawamata log terminal).
        \item $\lfloor \Delta \rfloor = 0$.
    \end{enumerate}
\end{proposition}
\begin{proof}
In general, $F$-regular log pairs are KLT and so $(a)\Longrightarrow (b)$ holds. Furthermore, note that every $P_i$ with $\delta_i=1$ in $\Delta=\delta_1P_1 + \cdots +\delta_r P_r$ is necessarily a log canonical center and an $F$-pure center of the pair $(X,\Delta)$. Thus, if $(X,\Delta)$ is KLT, then $\lfloor \Delta \rfloor = 0$ (i.e., $(b)\Longrightarrow (c)$ holds). The most interesting statement is the implication $(c)\Longrightarrow (a)$; which we do next.

For $(c)\Longrightarrow (a)$, we may assume that $f\:\bA^d \to X$ is a purely $\bQ$-factorial affine variety as in \autoref{subsec.QFactorialToricSingularities}. Then, using the transformation rule for the test ideals \cite{SchwedeTuckerTestIdealFiniteMaps,CarvajalStablerFsignaturefinitemorphisms}, we have 
\[
\bm{\tau}(X,\Delta) = \Tr(\bm{\tau}(\bA^d,f^* \Delta)) = \Tr(\bm{\tau}(\bA^d, \delta_1 \Div t_1 + \cdots + \delta_d \Div t_d)) =\Tr(\kay[t_1,\ldots,t_d])=\sO_X, 
\]
where $\Tr\: \kay[t_1,\ldots,t_d] \to \sO_X= \kay[t_1,\ldots,t_d]^G$ is the splitting trace map constructed in \cite{CarvajalFiniteTorsors}. The fact that $(\bA^d, \delta_1 \Div t_1 + \cdots + \delta_d \Div t_d)$ is $F$-regular precisely when $\lfloor \Delta \rfloor = 0$ follows from e.g. \cite[Example 4.18]{BlickleSchwedeTuckerFSigPairs1}.\footnote{A more in-depth analysis, as in \cite{CarvajalFiniteTorsors}, would show that $(\bA^d, f^* \Delta)$ is $F$-regular if and only if so is $(X,\Delta)$.}
\end{proof}

In particular, big divisors $E \in \FSupp(X)$ correspond to the KLT toric log pairs $(X,\Delta)$ of class index $1$, where the boundary $\Delta$ is unique up to numerical equivalence. This observation allows us to conclude the following.

\begin{corollary} \label{cor.ExistenceOfAmpleClassesInFrobSupport}
    Let $X$ be a $\bQ$-factorial toric variety. Then, there is an ample divisor $\llbracket E \rrbracket \in \FSupp(X)$ if and only if there is a toric log Fano pair $(X,\Delta)$ of class index $1$.
\end{corollary}

\begin{remark}
    Recall that a toric variety $X$ can always be endowed with a toric boundary $\Delta$ such that $(X,\Delta)$ is a toric log Fano pair. The problem is that in doing so, the index of $(X,\Delta)$ is potentially very large. Indeed, one starts with an ample toric divisor $A=a_1 P_1+\cdots + a_r P_r$ and then takes $n \gg 0$ such that
    \[
    \Delta \coloneqq -K_X - \frac{1}{n} A = \left(1-\frac{a_1}{n}\right) P_1 + \cdots +\left(1-\frac{a_r}{n}\right) P_r
    \]
    has coefficients in $(0,1)$. What is interesting about the toric log Fano structure \autoref{cor.ExistenceOfAmpleClassesInFrobSupport} is that at least the class index is $1$. 
\end{remark}

We further obtain the following useful fact.

\begin{proposition} \label{pro.BigElementsComeInPairs}
    Let $X$ be a $\bQ$-factorial toric variety and $\llbracket E \rrbracket \in \FSupp(X)$. Then, $E$ is big if and only if $-K_X-E$ is also in $\FSupp(X)$. In particular, if every big divisor in $\FSupp(X)$ is ample (resp. nef) then $X$ is Fano (resp. weak Fano).
\end{proposition}
\begin{proof}
    Let $(X,\Delta = \delta_1 P_1 + \cdots +\delta_r P_r)$ be a toric log pair corresponding to $E$. In particular,
    \[
    -(K_X+\Delta) \eqqcolon \Delta' \sim_{\bQ} E.
    \]
    Equivalently, we may express this as
    \[
    -(K_X+\Delta') = \Delta \sim_{\bQ} -K_X - E.
    \]

   If $E$ is big then we may take $\Delta$ such that $(X,\Delta)$ is KLT. In that case, $(X,\Delta' =(1-\delta_1) P_1 + \cdots +(1-\delta_r)P_r)$ is itself a KLT toric log pair and further $-K_X-E$ is in $\BFS(X)$. Conversely, if $\llbracket-K_X-E \rrbracket \in \FSupp(X)$, then we may take $\Delta$ such that $\Delta'$ is a toric boundary, and so $E$ is big. 

    For the last statement, let $E$ be a big divisor in $\FSupp(X)$. By what we just showed, $-K_X-E$ is another big divisor in $\FSupp(X)$. By hypothesis, we would then have that $E$ and $-K_X-E$ are ample (resp. nef). Adding them together implies that $-K_X$ is ample (resp. nef).
\end{proof}

\begin{remark}
    Observe that the above explains why $-K_X$ nor $0$ ever belong to $\FSupp(X)$. In general, we see that if $\llbracket E \rrbracket \in \FSupp(X)$ with toric boundary $\Delta$, then $\llbracket -K_X-E-\lceil \Delta \rceil \rrbracket \in \FSupp(X)$ with toric boundary $\{\Delta\}\coloneqq\Delta-\lceil\Delta \rceil$---the so-called \emph{fractional part} of $\Delta$.
\end{remark}

\begin{remark}
    Given $\llbracket E \rrbracket \in \BFS(X)$, so that $\llbracket E'\coloneqq - K_X-E \rrbracket \in\BFS(X)$, it is not necessarily true that $\tilde{m}(E;q)=\tilde{m}(E';q)$. However, we have $\alpha(E)=\alpha(E')$. 
\end{remark}

\subsubsection{The Frobenius support in derived categories}
    The negativity of $F^e_*\sO_X$, or equivalently the positivity of its dual, has also been studied in the context of derived categories. It started when A. Bondal claimed in \cite{bondal2006derived}, without proof, that for a smooth toric variety $X$, the line bundles corresponding to the divisor classes in $\FSupp(X)$, together with the structure sheaf $\sO_X$, generate $D^b(X)$, the bounded derived category of coherent sheaves over $X$. Let us denote this set of generators of $D^b(X)$ by $\sC$. Bondal further argued that if $(F^e_*\sO_X)^{\vee}$, or equivalently $\sE_{X,e}$, is nef, then $\sC$ is a full strong exceptional collection. 
    
    Some years later, H.~Uehara proved Bondal's first claim \cite{UeheraExceptionalCollectionsOnToricFanoThreefolds}---which was known as Bondal's conjecture---for toric Fano threefolds and constructed full strong exceptional collections for each of them. When $\sE_{X,e}$ is not nef, the idea is to carefully choose a subset $\sS$ of $\sC$ that forms a strongly exceptional collection and then to show that $\sS$ and $\sC$ generate the same category. This method was also used in \cite{DerivedCategoryOfToricVarietiesWithPicardNumberThree,DerivedCategoryOfToricVarietiesWithSmallPicardNumber}. However, as shown by A.~I.~Efimov \cite{EfimovCounterExampleToKingsConjectureForToricFano}, there are toric Fano varieties that do not admit a full strongly exceptional collection of invertible sheaves, so Uehara's method does not always work. 
    
    Bondal's conjecture was recently proven in full generality by Hanlon, Hicks, and Lazarev \cite{HanlonHicksLazarevResolutionsOfToricSubvarietiesByLineBundles}. In particular, \autoref{thmNefnessOfE} below implies that the extremal toric Fano varieties admit a full strong exceptional collection of line bundles. An interesting consequence of this is the following corollary: 
\begin{corollary}
    Let $X$ be a smooth toric variety with $\sE_{X,e}$ nef for all $e>0$, or equivalently, let $X$ be a toric extremal Fano variety. Then, $|\FSupp(X)| = s-1 = |\Sigma(d)| - 1$.  
\end{corollary}
\begin{proof}
    By the discussion above, if $\sE_{X,e}$ is nef, the line bundles $\{\sO_X\} \cup \{\sO_X(E)\}_{[E] \in \FSupp(X)}$ form a full strong exceptional collection. The length of a full exceptional sequence is equal to the rank of the K-group $K_0(X)$ which, for smooth toric varieties, is equal to $|\Sigma(d)|$, see for instance the discussion in \cite[1.1]{AltmannWittTheStructureOfExceptionalSequencesToricVarieties} 
\end{proof}

\section{On the Bigness of the Frobenius-Trace Kernel} \label{sec.BIGNESS}

 In this section, we prove the following result, and thus Theorem A in the Introduction.

\begin{theorem} \label{thm.MainTheoremBignessAmplenessProjSpace}
    Let $X$ be a smooth $d$-dimensional toric variety such that $\sE_{X,e}$ is big for all $e \gg 0$, i.e., $\BFS(X)=\FSupp(X)$. Then, $X \simeq \bP^d$.
\end{theorem}

\begin{proof}
    Suppose that $X \not\simeq \bP^d$ and so $\rho=\rho(X)\geq 2$. By \autoref{cor.SupportOfE_X}, we must show that 
    \[
    (\sQ_X \cap N^1(X) \setminus 0) \cap \partial \eff(X) \neq \emptyset.
    \]
    
Since $\eff(X)$ is a strongly convex rational cone (see \autoref{prop.PseudoEffectiveCone}), then so are each of its faces. Then, by \autoref{lemma.FindingPlatticePoints}, it suffices to find a face $F$ of $\eff(X)$ containing $\pi_1,\dots,\pi_k$ with $k \geq \dim F+1 =\rho$ (possibly after relabeling the prime toric divisors). Indeed,  \autoref{lemma.FindingPlatticePoints} then implies the existence of a nonzero lattice point in $\sQ_F \coloneqq \langle \pi_1, \ldots, \pi_k \rangle_{[0,1)} 
      \subset \sQ_X$. However, $\sQ_F \subset F \subset \partial \eff(X)$.
 
Given an element $0\neq w \in \mov(X)$, the set $F_w \coloneqq \{v \in \eff(X) \;|\; \langle v,w\rangle = 0\}$ is a face of $\eff(X)$. In our setting, the elements of $\mov(X)$ correspond to the effective linear relations between the primitive ray generators of $\Sigma=\Sigma_X$. Take such an element to be a centrally symmetric primitive relation as in \autoref{Rem.CentrallySymmetricPrimitiveRelations}. Say,
\[
  R\coloneqq R_{\sP}\: u_1 +\cdots + u_k = 0,
\] 
where $k \leq d$ as $X \not\simeq \bP^d$; see \autoref{prop.ToricCharacterizationOfProjectiveSpace}. Then, there are at least $\rho$ primitive ray generators of $\Sigma$ not appearing in $R$; say $u_{k+1},\ldots,u_{k+\rho}$. In particular, $\pi_i \cdot R = 0$ for all $i\in \{k+1,\ldots,k+\rho\}$. In other words, the face $F_R$ of $\eff(X)$ defined by $R \in \mov(X)$ contains $\pi_{k+1}, \ldots, \pi_{k+\rho}$. Note that $F_R$ is a proper face as it does not contain $\pi_1$. Since $\dim \eff(X) = \rho$, then $\dim F_R < \rho$. This finishes the proof.
\end{proof}

\begin{corollary} \label{cor.FrobeniusAndMori}
    Let $X$ be a smooth toric variety. The following statements are equivalent.
\begin{enumerate}
    \item $\sT_X$ is ample.
    \item $X$ is a projective space.
    \item $\sE_{X,e}$ is ample for all $e>0$.
\end{enumerate}
\end{corollary}

\begin{remark}
    The following observations are in order with respect to \autoref{cor.FrobeniusAndMori}:
\begin{enumerate}
    \item The authors do not know whether a smooth toric variety with big tangent sheaf must be a projective space. See \cite{WuToricMoriTheorem}.
    \item The quantifier on $e$ may be replaced by ``for some $e>0$'' if $\dim X \leq 3$. See \cite{CarvajalRojasPatakfalviVarietieswithAmpleFrobenius-traceKernel}. The point is that if $\sE_{X,e}$ is ample for some $e>0$ then fibrations and smooth blowups are immediately ruled out. In low dimensions, this is enough to conclude that $X$ is a projective space. In dimensions $\geq 4$, small contractions seem to be an issue.
\end{enumerate}
\end{remark}

Suppose that, in the proof of \autoref{thm.MainTheoremBignessAmplenessProjSpace}, the centrally symmetric primitive relation $R$ happens to be extremal, i.e., if we take an extremal primitive relation of the form $R\: b_1u_1 + \cdots +b_k u_k=0$ so that $k\leq d$ (\autoref{thm.MoriTheoryOfFanoVarieties}).  Then the same argument gives the following result, which we will describe more geometrically in \autoref{subsubsection.MorFibrationFeffectiveness}.

\begin{scholium} \label{scholium.BignessAndFibrations}
    Let $X$ be a $\bQ$-factorial toric variety. If $X$ admits a Mori fibration $X\to S$ with $\dim S \neq 0$, then there is a non-big class in the Frobenius support of $X$. In fact, there is an element of $\FSupp(X)$ in every facet of $\eff(X)$ that contains a facet of the moving cone of divisors and in particular of $\nef(X)$.
\end{scholium}

\begin{remark}[On the $\bQ$-factorial case]
If $X$ is a weighted projective space, then $\sE_{X,e}$ is ample. In particular, \autoref{thm.MainTheoremBignessAmplenessProjSpace} does not hold in the singular case. More generally, we readily see that if $\rho(X)=1$ then $\sE_{X,e}$ is ample for all $e>0$. The $\bQ$-factorial toric varieties with Picard rank $1$ are the so-called fake weighted projective spaces, which are quotients of weighted projective spaces by finite toric subgroups. For more on this type of prime Fano varieties, see \cite{KaprzykFakeWeightedProjectiveSpaces}. We may wonder whether, for a $\bQ$-factorial toric variety $X$, its Frobenius support being big implies that $X$ is a prime Fano variety, i.e., $\rho(X)=1$. This would follow \emph{verbatim} as in the proof of \autoref{thm.MainTheoremBignessAmplenessProjSpace} if we were granted the  existence of a primitive relation of the form $R\: b_1 u_1 + \cdots + b_k u_k=0$ with $b_1,\ldots,b_k \geq 0$ and $k\leq d$. This is a particular type of nef primitive relation as coined by Rossi and Terracini in their classification work for $\bQ$-factorial toric varieties; see \cite{RossiTerraciniClassificationofQFactorialToricVarieties}. However, as they explain in their work, nef primitive relations may fail to exist beyond the smooth case. We will see below in \autoref{example.FatalExample} an example of a singular $\bQ$-factorial toric variety of Picard rank $2$ whose Frobenius support is big and nef.
\end{remark}

\section{On the Numerical Effectiveness of the Frobenius-trace Kernel} \label{sec.Nefness}

In this section, we study the interaction between the Frobenius support and the cone of moving divisors. We start by describing intersection number in terms of Frobenius.

\subsection{Intersection numbers and the smooth case} We are ready to prove Theorem B from \autoref{sec.Intro}. For the reader's convenience, we break it down into smaller statements.

\begin{proposition}
     Let $X$ be a $\bQ$-factorial toric variety and 
    \[
        R \: b_1u_1 + \dots + b_ku_k = 0
    \]
    be a primitive minimal effective relation in $N_1(X)$, i.e., the coefficients $0\neq b_1,\ldots,b_k \in \bN$ have no common factor and $R$ spans an extremal ray of $\mov(X)$ and so it corresponds to a facet of $\eff(X)$. For every $i=1,\ldots ,k$, there is $\llbracket E_i \rrbracket \in \BFS(X)$ such that $E_i \cdot R = \pi_i \cdot R=  b_i$.
\end{proposition}
\begin{proof}
    We may assume that the classes $\pi_{k+1},\ldots,\pi_{k+\rho-1} \in N^1(X)$ are linearly independent, for $F_R \coloneqq \langle \pi_{k+1},\ldots,\pi_r \rangle_{\bR_{\geq 0}}$ is the facet of $\eff(X)$ cut out by $R$. There is a non-trivial relation
    \[
    c_{1} \pi_{1} +  c_{2} \pi_{2} +  c_{k+1} \pi_{k+1} + \cdots +  c_{l+\rho-1} \pi_{l+\rho-1}=0
    \]
    where $c_{1}, c_{2}, c_{k+1}, \ldots, c_{l+\rho - 1} \in \bZ$ and $(c_1,c_2) \neq 0$. Intersecting it with $R$ yields the equality
    \[
    0=c_{1}b_{1}+c_{2}b_{2}.
    \]
    In particular, $c_{1}c_{2}<0$ as $b_1 b_2 >0$. By \autoref{prop.PseudoEffectiveCone}, \autoref{lemma.FindingPlatticePoints}, and  \autoref{cor.SupportOfE_X}, we obtain that
    \[
    \llbracket E_{1} \rrbracket \coloneqq \varepsilon_1 \coloneqq  \sum_{i \: c_i c_{1} >0} \pi_i \in \FSupp(X).
    \]
    Moreover, $E_{1} \cdot R = b_{1}$. Of course, there is nothing special about $i=1$ and the statement for the other indices follows after relabeling.

    It remains to show that $\varepsilon_1$ can be taken to be big. To this end, we look at $l = \dim \langle \pi_1,\ldots,\pi_k \rangle_{\bR} \geq 1$. If $l \geq 2$, then in the above argument we could have taken $\pi_1$ and $\pi_2$ to be linearly independent. This means that $\varepsilon_1 - \pi_1 \eqqcolon \delta \in F_R \setminus 0$. In particular, $\varepsilon_1= \pi_1 + \delta$ is big as $\pi_1 \notin F_R$.

    Suppose now that $l = 1$. Then, $\xi \coloneqq \pi_1/b_1 = \cdots = \pi_k/b_k$. Moreover, $b_i \xi = \pi_i = \varepsilon_i \in \FSupp(X)$ for all $i=1,\ldots,k$. Observe that $\xi$ spans the only ray of $\eff(X)$ not contained in $F_R$, and so $\xi$ is not big. Assuming $b_1\leq \cdots \leq b_k$, the above implies that $b_1 | b_2 | \cdots | b_k$. In particular, $b_1 = 1$ and $\xi = \pi_1 \in \FSupp(X)$.
    
    Let $\beta \in \BFS(X)$, which exists by \autoref{cor.multiplicityOfDirectSummands}. Let us write $\beta = c_1 \pi_1 + \cdots + c_r \pi_r$ for some $0\leq c_1,\ldots,c_r < 1$. Then,
    \[
    \beta = (\beta \cdot R) \xi + c_1 \pi_{k+1} +\cdots + c_r \pi_r
    \]
    where $\beta \cdot R = c_1 b_1 + \cdots +c_k b_k$. Observe that $0 \neq \beta \cdot R \in \bN$ as $\beta$ is also an integral linear combination of the $\pi_1,\ldots,\pi_r$ whose intersection numbers against $R$ are integral by hypothesis. 
   
    With the above in place, for $i=1,\ldots,r$, we may write
    \[
    \beta_i \coloneqq \beta - (\beta \cdot R - b_i) \xi = \pi_i + c_{k+1} \pi_{k+1} + \cdots +c_r \pi_r = c_1'\pi_{1} + \cdots + c'_i \pi_k + c_{k+1} \pi_{k+1} + \cdots +c_r \pi_r, 
    \]
    where $c_i' \coloneqq b_i/(b_1+\cdots+b_k) \in (0,1)$. The first, defining equality shows that $0 \neq \beta_i \in N^1(X)$ as $\beta,\xi \in N^1(X)$ and $\beta$ is big while $\xi$ is not. The last equality shows that $\beta_i \in \mathcal{Q}_X^{\circ}$. By \autoref{cor.SupportOfE_X}, this means that $\beta_i \in \BFS(X)$. Moreover, $\beta_i \cdot R =b_i$.
    \end{proof}

\begin{proposition} \label{prop.NefnessSmallContractions}
    Let $X$ be a $\bQ$-factorial toric variety. Suppose that $X$ admits a small extremal contraction given by an extremal primitive relation
    \[
        R \: b_1u_1 + \dots + b_ku_k - b_{k+1}u_{k+1} - \dots - b_lu_l = 0
    \]
    with $l \geq k+2$. For every $i=1,\ldots ,l$, there is $\llbracket E_i \rrbracket \in \FSupp(X)$ such that $E_i \cdot R = \pi_i \cdot R$. In particular, $\sE_{X,e}$ is not nef for any $e\gg 0$. Moreover, the following statements hold:
    \begin{enumerate}
        \item For the indices $i$ for which $\pi_i$ is big, we may further take $E_i$ to be big.
        \item If $\langle \pi_{k+1},\ldots,\pi_{k+\rho-1} \rangle_{\bR_{\geq 0}}$ is a facet of $\eff(X)$, we can take $E_i$ to be big for all $i=k+1,\ldots,k+\rho-1$.
        \item If $\langle \pi_{1},\ldots,\pi_{\rho-1} \rangle_{\bR_{\geq 0}}$ is a facet of $\eff(X)$, we can take $E_i$ to be big for all $i=1,\ldots,\rho-1$.
    \end{enumerate}
    
\end{proposition}
\begin{proof}
   By \autoref{keyLemma}, we may assume that the prime classes $\pi_{l+1},\ldots, \pi_{l+\rho -1}$ are linearly independent. Then, there is a non-trivial relation
    \[
    c_{k+1} \pi_{k+1} +  c_{k+2} \pi_{k+2} +  c_{l+1} \pi_{l+1} + \cdots +  c_{l+\rho-1} \pi_{l+\rho-1}=0
    \]
    where $c_{k+1}, c_{k+2}, c_{l+1}, \ldots, c_{l+\rho - 1} \in \bZ$ and $(c_{k+1},c_{k+2}) \neq (0,0)$. Intersecting this relation with $R$ yields the equality
    \[
    0=-c_{k+1}b_{k+1}-c_{k+2}b_{k+2}.
    \]
    In particular, $c_{k+1}c_{k+2}<0$ as $b_{k+1} b_{k+2} >0$. By \autoref{lemma.FindingPlatticePoints} and  \autoref{cor.SupportOfE_X}, we have
    \[
    \llbracket E_{k+1} \rrbracket \coloneqq \varepsilon_{k+1} \coloneqq \sum_{i \: c_i c_{k+1} >0} \pi_i \in \FSupp(X).
    \]
    Note that this also relies on \autoref{prop.PseudoEffectiveCone}. Moreover, $ E_{k+1} \cdot R = -b_{k+1}$. Of course, there is nothing special about $i=k+1$ and the statement for the other indices $\geq k+2$ follows by relabeling. For the indices $\leq k$, one does a flip and uses \autoref{cor.FrobeniusSupport}.

It remains to prove the claims (a) and (b). To this end, write
\[
\varepsilon_{k+1} = \pi_{k+1} + \delta
\]
with $\delta$ being a certain sum of $\sum_{i\in I} \pi_i $ with $I \subset \{l+1,\ldots,l+\rho-1\}$. Since $\delta \in \eff(X)$, it follows that $\varepsilon_{k+1}$ is big if so is $\pi_{k+1}$. This shows (a).

For (b), assume that the $\pi_{k+1}, \ldots, \pi_{k+\rho-1}$ span a facet of $\eff(X)$. Then, $\pi_{k+1}$ and $\pi_{k+2}$ are linearly independent and so $\delta$ is a nonzero element in $\eff(X) \cap R^{\perp}$.  Therefore, $\varepsilon_{k+1}=\pi_{k+1} + \delta$ must be big, as otherwise $\pi_{k+1}$ and $\delta$ would have to be in the same facet, which is absurd.

Finally, (c) is just the flipped version of (b).
\end{proof}

\begin{corollary} \label{cor.KetCorollary}
With notation as in \autoref{prop.NefnessSmallContractions}, suppose that $\dim \langle \pi_{k+1},\ldots,\pi_{l}\rangle_{\bR} \geq \rho -1$. Then there is $\llbracket E \rrbracket \in \BFS(X)$ such that $E \cdot R < 0$.    
\end{corollary}
\begin{proof}
   By (a) in \autoref{prop.NefnessSmallContractions}, we may assume that $\pi_{k+1},\ldots,\pi_{l} \in \partial \eff(X)$. The hypothesis on the dimension then implies that there are $\rho-1$ vectors in $\{\pi_{k+1},\ldots,\pi_{l}\}$ spanning a facet of $\eff(X)$. Then one may apply (b) in \autoref{prop.NefnessSmallContractions}.
\end{proof}

\begin{corollary}
    Let $X$ be a smooth toric variety such that $E \cdot C \geq -1$ for all $\llbracket E \rrbracket \in \FSupp(X)$ and curves $C \subset X$. Then, all toric $\bQ$-factorial modifications of $X$ are smooth.
\end{corollary}

\begin{proposition} \label{prop.NonSmoothBlowups}
    Let $X$ be a $\bQ$-factorial toric variety. Suppose that $X$ admits a divisorial extremal contraction given by an extremal primitive relation
    \[
        R \: b_1u_1 + \dots + b_ku_k - b_{k+1}u_{k+1} = 0.
    \]
Then, for every $i=1,\ldots,k$ there is $\llbracket E_i \rrbracket \in \FSupp(X)$ such that $E_i \cdot R = b_i-b_{k+1}$. In particular, if $\sE_{X,e}$ is nef (resp. ample) for all $e\gg 0$ then $b_{k+1} \leq b_i$ (resp. $b_{k+1} < b_i$) for all $i=1,\ldots,k$.
\end{proposition}
\begin{proof}
By \autoref{keyLemma}, we may assume that the classes $\pi_{k+2},\ldots, \pi_{k+\rho }$ are linearly independent. In particular, there is a non-trivial relation
    \[
    c_{1} \pi_{1} + c_{k+1} \pi_{k+1} + c_{k+2} \pi_{k+2} + \cdots +  c_{k+\rho} \pi_{k+\rho} = 0
    \]
    where $c_1,c_{k+1}, \ldots, c_{k+\rho} \in \bZ$ and $(c_1,c_{k+1}) \neq (0,0)$. Intersecting this relation with $R$ yields
    \[
    0=c_1b_1-c_{k+1}b_{k+1}.
    \]
    In particular, $c_1 c_{k+1} >0$ as $b_1b_{k+1}>0$. Applying \autoref{lemma.FindingPlatticePoints} and \autoref{cor.SupportOfE_X} (as well as  \autoref{prop.PseudoEffectiveCone}), we conclude that
    \[
    \llbracket E_1 \rrbracket \coloneqq \sum_{i \: c_i c_{k+1} >0} \pi_i \in \FSupp(X).
    \]
    Observe that $ E_1 \cdot R = b_1-b_{k+1}$. The cases $i=2,\ldots,k$ follow by relabeling. 
\end{proof}

\begin{corollary} [{\cf \cite{CarvajalRojasPatakfalviVarietieswithAmpleFrobenius-traceKernel}}] \label{scholium.SmoothBlowudownsFrobeSupport}
        Let $X$ be a smooth toric variety. Suppose that $X$ admits an extremal divisorial contraction given by an extremal primitive relation  
        \[
        R \: u_1 + \dots + u_k - a_{k+1} u_{k+1} = 0.
        \]
        The following statements hold:
\begin{enumerate}
    \item   If the divisorial contraction is not a smooth blowup (i.e., $a_{k+1}>1$) then $\sE_{X,e}$ is not nef for any $e \gg 0$.
\item   If the divisorial contraction is a smooth blowup (i.e., $a_{k+1}=1$). Then there is $\llbracket E \rrbracket \in \FSupp(X)$ such that $E \cdot R = 0$. In particular, $\sE_{X,e}$ is not ample for any $e\gg 0$.

\end{enumerate}
\end{corollary}

\begin{remark} 
We will see below in \autoref{cor.FSIntersectInternalFacetMovingCone} that $E$ in \autoref{scholium.SmoothBlowudownsFrobeSupport} may be taken big. However, it is unclear to the authors whether it can always be taken nef. 
\end{remark}

\begin{remark}
    Putting together \autoref{scholium.BignessAndFibrations}, \autoref{prop.NefnessSmallContractions}, and \autoref{scholium.SmoothBlowudownsFrobeSupport}, we see that a smooth toric variety with ample Frobenius support must be a projective space without using the existence of centrally symmetric primitive relations. Further, a $\bQ$-factorial toric variety with big and nef Frobenius support admits only divisorial extremal contractions.
\end{remark}

\begin{proposition}\label{prop.AllDivisorialContractionsSmoothBlowUpImpliesENef}
   Let $X$ be a $\bQ$-factorial toric variety and $R$ be an extremal primitive relation defining an extremal contraction $\phi$. The following statements hold:
   \begin{enumerate}
       \item If $\phi$ is a Mori fibration then $E \cdot R \geq 0$ for all $\llbracket E \rrbracket \in \FSupp(X)$.
       \item If $X$ is smooth and $\phi$ is a smooth blowup then $E \cdot R \geq 0$ for all $\llbracket E \rrbracket \in \FSupp(X)$.
       \item In particular, $X$ is an extremal Fano variety then $\sE_{X,e}$ is nef for all $e>0$
   \end{enumerate}
\end{proposition}
\begin{proof}
    Let $\llbracket E \rrbracket \in \FSupp(X)$ and write $\llbracket E \rrbracket=c_1 \pi_1+\cdots +c_r \pi_r$ with $c_1,\ldots,c_r \in [0,1)$. If $R \: u_1 + \cdots + u_k=0$, then $E \cdot R = c_1 + \cdots + c_k \geq 0$. This proves (a). For (b), note that if $R$ is of the form $R \: u_1 + \cdots + u_k - u_{k+1}=0$ then
    \[
   \bZ \ni E \cdot R = c_1 + \cdots + c_k -c_{k+1} > -1,
    \]
    where we use that $\llbracket E \rrbracket \in N^1(X)$ together with smoothness/factoriality to say that $E \cdot R \in \bZ$. Therefore, $E \cdot R \geq 0$. Statement (c) follows at once from the former two.
\end{proof}

Putting everything together, we can deliver the promised Theorem B from \autoref{sec.Intro}.

\begin{theorem}\label{thmNefnessOfE}
Let $X$ be a smooth toric variety. Then, $\sE_{X,e}$ is nef for all $e > 0$ if and only if $X$ is an extremal Fano variety. In that case, $\FSupp(X)$ intersects every facet of $\nef(X)$.
\end{theorem}

\begin{remark}
    We obtain a new proof of \autoref{cor.FrobeniusAndMori} using \autoref{thmNefnessOfE} in conjunction with \cite[Propositions 4.2 and 5.12]{CarvajalRojasPatakfalviVarietieswithAmpleFrobenius-traceKernel}. 
\end{remark}

\subsubsection{Blowing down to homogeneous spaces} For \autoref{thmNefnessOfE} to reach its full potential, we should be able to keep track of the positivity of the Frobenius-trace kernels when we perform these simplest Mori contractions. We explain how to achieve this next.

\begin{lemma} \label{lemma.DescendNefnessFibrations}
    Let $f\:X \to S$ be a smooth proper morphism admitting a section $i \: S \to X$. Then, if $\sE_{X,e}$ is nef, then so is $\sE_{S,e}$. 
\end{lemma}
\begin{proof}
     We have a surjective morphism $\varepsilon_{f,e} \: \sE_{X,e} \to f^* \sE_{S,e}$; see \cite[Proposition 2.4]{CarvajalRojasPatakfalviVarietieswithAmpleFrobenius-traceKernel}. Pulling it back along $i$ then yields a surjection $i^* \sE_{X,e} \to \sE_{S,e}$ which shows the claimed result for $i$ is a closed immersion as $f$ is separated (see \cite[\href{https://stacks.math.columbia.edu/tag/01KT}{Tag 01KT}]{stacks-project}).
\end{proof}

\begin{lemma} \label{lemma.DescendNefnessBlowdowns}
    Let $f\: X \to S$ be a smooth blowup. If $\sE_{X,e}$ is split and nef then so is $\sE_{S,e}$.
\end{lemma}
\begin{proof}
    One observes that the pushforward of $\kappa_X^e \: F^e_* \omega_X \to \omega_X$ is precisely $\kappa_S^e \: F^e_* \omega_S \to \omega_S$. In particular, letting $E\subset X$ be the exceptional divisor and $c$ the codimension of the center of $f$, it follows that
    \[
    \sE_{S,e} \otimes \omega_S = f_*(\sE_{X,e} \otimes \omega_X)=f_*(\sE_{X,e} \otimes f^* \omega_S \otimes \sO_X((c-1)E)) = f_*(\sE_{X,e} \otimes \sO_X((c-1)E)) \otimes \omega_S
    \]
    and so
    \[
    \sE_{S,e} =  f_*(\sE_{X,e} \otimes \sO_X((c-1)E)).
    \]
    
    By hypothesis, we may write a decomposition
    \[
    \sE_{X,e} = \bigoplus_{i=1}^{q-1} \sM_i, \quad   \sM_i = f^*\sL_i \otimes \sO_X(n_i E)
    \]
    where the $\sL_i$ are invertible sheaves on $S$ and $n_i \in \bZ$. Since $\sE_{X,e}$ is assumed to be nef, we may further say that $n_i \leq 0$ (by, say, intersecting $ \sM_i=f^*\sL_i \otimes \sO_X(n_i E)$ with a contracted curve inside $E$).

    On the other hand,  
    \[
    f_* \sM_i((c-1)E) =  \sL_i \otimes f_*\sO_X((c-1+n_i)E)
    \]
    is a direct summand of $\sE_{S,e}$ and, in particular, an invertible sheaf. However, since \[
    f_* \sO_X(mE)=\begin{cases}
        \sI^m & \text{for } m<0,\\
        \sO_S & \text{otherwise},
    \end{cases}
    \]
    where $\sI \subset \sO_X$ is the prime ideal sheaf being blown up, this implies that $0 \geq n_i \geq -(c-1)$ and further that $f_* \sM_i((c-1)E) =  \sL_i$. In particular, $\sE_{S,e} = \bigoplus_{i=1}^{q-1} \sL_i$ and so all we have to do is to observe that $\sL_i$ must be nef. Indeed, letting $C \subset S$ be a curve and $C' \subset X$ be its strict transform, we have $E \cdot C'\geq 0$ and so
    \[
    \sL_i \cdot C = f^* \sL_i \cdot C'\geq f^* \sL_i \cdot C' + n_iE \cdot C' = \sM_i \cdot C' \geq 0;
    \]
    as required.
\end{proof}

\begin{remark}
    The authors do not know whether the splitting hypothesis on $\sE_{X,e}$ in \autoref{lemma.DescendNefnessBlowdowns} can be dropped.
\end{remark}

\begin{scholium} \label{scholium.SmoothBlowupsFTKDescription}
Let $f\: X \to S$ be a smooth blowup between toric varieties. Let $f_{!} \: \Cl(X) \to \Cl (S)$ be the canonical retraction of $f^* \:\Cl(S) \to \Cl(X)$ that leads to the decomposition $\Cl(X) = \Cl(S) \oplus \langle E \rangle_{\bZ}$ where $E\subset X$ is the exceptional divisor of $f$. Then $f_{!}$ induces a surjection $
f_{!} \:\FSupp(X) \to \FSupp(S) $ where nef (resp. ample) classes are sent to nef (resp. ample) classes.\footnote{This also uses that on toric varieties ampleness and strict nefness are equivalent notions.} Moreover,
\[
\sE_{X,e} = \bigoplus_{\substack{D \in \FSupp(S) \\ n\geq -(c-1)}} \sO_X(f^* D+ n E)^{\oplus m(D,n;q)}
\]
where $\sum_{n} m(D,n;q)=m(D;q)$. Hence, $\alpha(D) = \sum_{n \geq -(c-1)} \alpha(f^*D + n E )$ and, in particular, $f_{!} \:\FSupp(X) \to \FSupp(S)$ sends big classes to big classes.
\end{scholium}

\begin{corollary} \label{cor.descendingNefness}
    Let $X$ be a toric extremal Fano variety and $X \to S$ be a Mori contraction. Then $S$ is a toric extremal Fano variety.
\end{corollary}

\begin{corollary} \label{cor.NefnessChainSmoothBlowyups}
Let $X$ be a smooth toric variety such that $\sE_{X,e}$ is nef for all $e>0$. Then, there is a finite chain of smooth blowups
\[
X \to X_1 \to X_2 \to X_3 \to \cdots \to X_n
\]
such that $X_n$ is a homogeneous space. In particular, $n < \rho(X)$.
\end{corollary}

\begin{remark}[Analogy with Campana--Peternell's conjecture]
    The equivalence between the homogeneity of a smooth variety $X$ and the global generation of its tangent sheaf holds as long as $\Aut(X)$ is reduced; e.g., in characteristic zero. See \cite[Proposition 2.1]{MunozOcchettaGianlucaWatanabeWizniweskiCampanaPeternellSurvey}. Examples of homogeneous spaces include abelian varieties, which have a trivial tangent sheaf. Another one is given by \emph{rational} homogeneous spaces, i.e., quotients of semi-simple Lie groups by parabolic subgroups. Any other homogeneous space is a product of homogeneous spaces of this kind. The rational homogeneous spaces are the Fano homogeneous spaces. Campana--Peternell's conjecture asserts that a Fano variety with nef tangent sheaf must be a rational homogeneous space. For more on this wonderful problem, see \cite{MunozOcchettaGianlucaWatanabeWizniweskiCampanaPeternellSurvey}. As said in the Introduction, the toric varieties confirm this conjecture. This raises two questions:
    \begin{enumerate}
        \item What can be said about the Frobenius-trace kernels on homogeous spaces?
        \item What strengthening on nefness for Frobenius-trace kernels characterizes homogeneity for toric varieties?
    \end{enumerate}
    We will pursue the first question elsewhere. For the second, inspired by the notion of $F$-signature from the theory of $F$-singularities, we will work on this in \autoref{sec.AmpleFSignHomogeneity}.
\end{remark}

\subsection{The general $\bQ$-factorial case}

\autoref{thmNefnessOfE} characterizes when $\FSupp(X) \subset \nef(X)$, at least in the smooth case. We aim to also understand the $\bQ$-factorial case, even if the answer will not be as exact. To do so, we will examine first when \emph{$\FSupp(X)$ moves}, i.e., 
 \[
 \FSupp(X) \subset \mov^{1}(X).
 \]
To this end, the following notion will be very useful.

\begin{definition} \label{def.InertStuff}
    Let $X$ be a $\bQ$-factorial toric variety. We say that $X$ is \emph{divisorially inert} if all its divisorial extremal contractions are inert. Furthermore, $X$ is said to be \emph{birationally inert} if all its birational extremal contractions are inert divisorial contractions.
\end{definition}

\begin{example}
     In the smooth case, an inert divisorial extremal contraction is the same as a smooth blowdown. Hence, a smooth toric variety is birationally inert if and only if it is an extremal Fano variety.
\end{example}

The next key notion is that of \emph{Frobenius effective curves and divisors}.

\begin{definition}[Frobenius effective cycles]
    Let $X$ be a toric variety. We define:
    \begin{enumerate}
        \item The \emph{$F$-effective cone} of $X$ is
        \[
        \Frob(X) \coloneqq \langle \FSupp(X) \rangle_{\bR_{\geq 0}} \subset \eff(X). 
        \]
        \item A divisor is said to be \emph{$F$-effective} if its class belongs to $\Frob(X)$.
        \item The \emph{cone of $F$-effective curves} of $X$ is
        \[
        \FE(X) \coloneqq  \Frob(X)^{\vee},
        \]
        and refer to its elements as \emph{$F$-effective $1$-cycles}.
    \end{enumerate}
\end{definition}

\begin{example}
    By \autoref{cor.multiplicityOfDirectSummands} and \autoref{pro.BigElementsComeInPairs}, the canonical divisor $-K_X$ is $F$-effective. In fact, these further show that $-K_X$ sits in the interior of $\Frob(X)$.
\end{example}

Observe that $\mov(X) \subset \FE(X)$, in particular, $\dim \FE(X) = \rho(X)$ and so $\Frob(X)$ is a strongly convex rational cone. We shall see below in \autoref{cor.DimensionOfFrobeniusCone} that $\dim \Frob(X) = \rho(X)$ and so $\FE(X)$ is strongly convex as well.

Our next task is to analyze the $F$-effectiveness of extremal rays of the Mori cone. It turns out that birational and fibration extremal rays exhibit opposite behavior. 

\subsubsection{Birational extremal contractions and $F$-effectiveness}

Let $\phi: X \to S$ be an extremal divisorial contraction as in \autoref{prop.LocalStructureExtremalDivContraction}. Consider the exact sequence
\[
     0 \to \langle [P_{k+1}] \rangle_{\bZ} \to \Cl(X) \xrightarrow{\phi_!} \Cl(S) \to 0
     \]
     where we use the canonical isomorphisms
     \[
       \Cl(S) = \Cl(S\setminus C) = \Cl(X \setminus P_{k+1}).
      \]
      This induces the exact sequence
       \begin{equation} \label{eqn.SESNeronSeveriSapcesDivContraction}
               0 \to \langle \pi_{k+1} \rangle_{\bR} \to N^1(X)_{\bR} \xrightarrow{\phi_!} N^1(S)_{\bR} \to 0.
       \end{equation}
Observe that $\phi_!(\eff(X)) = \eff(S)$. Furthermore, $\phi_! \varepsilon$ is nef (resp. ample) if so is $\varepsilon$. Likewise, since $\phi_!$ is open, $\phi_! \varepsilon$ is big if so is $\varepsilon$.
     
      The pullback map $\phi^* \:N^1(S)_{\bR} \to N^1(X)_{\bR}$ provides a canonical splitting of \autoref{eqn.SESNeronSeveriSapcesDivContraction} so that
      \begin{equation} \label{eqn.NeronSeveriDivisorialContractions}
           N^1(X)_{\bR} =\langle \pi_{k+1}\rangle_{\bR} \oplus \phi^*(N^1(S)_{\bR})
      \end{equation}
      where 
      \[
      \phi^*(N^1(S)_{\bR}) = \langle \pi_{k+1},\ldots,\pi_{k+\rho} \rangle_{\bR} = R^{\perp} \subset N^1(X)_{\bR}
      \] 
       is the subspace of numerical classes that do not intersect $R$; see \autoref{proposition.nefConeContainedInSomeCone}. 
       In particular, for every $\varepsilon \in N^1(X)_{\bR}$ there is a unique $\delta  \in N^1(S)_{\bR}$ such that 
      \[
      \varepsilon = -(\varepsilon \cdot R/b_{k+1}) \pi_{k+1} + \phi^* \delta.
      \]

We may rephrase and generalize \autoref{prop.NefnessSmallContractions} and \autoref{prop.AllDivisorialContractionsSmoothBlowUpImpliesENef} as follows.

\begin{proposition}[$F$-effectiveness of birational extremal contractions] \label{MainProposition}
    Let $X$ be a $\bQ$-factorial toric variety and $R$ be an extremal primitive relation defining a birational extremal contraction $\phi\: X \to S$. The following statements hold.
    \begin{enumerate}
        \item If $\phi$ is a small contraction, then $R \notin \FE(X)$.
        \item If $\phi$ is a divisorial contraction, then $R \in \FE(X)$ if and only if $\phi$ is inert. 
        \item Suppose that $\phi$ is an inert divisorial contraction. Then:
        \begin{enumerate}
            \item  For a given $\delta \in \FSupp(S)$, the set of values $ \lfloor \tilde{D} \cdot R \rfloor$, where $\tilde{D}$ is the strict transform of a $\bT$-invariant $[0,1)$-divisor $D$ on $S$ such that $\llbracket D \rrbracket = \delta$, is a discrete interval $J_{\delta} \coloneqq [i_\delta,k_{\delta}] \cap \bN \subset [0,b_1+\cdots+b_k-1]$.
            \item With notation as above, $\phi_{!}$ induces a surjection $\phi_{!} \:\FSupp(X) \twoheadrightarrow \FSupp(S)$  whose fiber at $\delta \in \FSupp(S)$ is 
  $\{-j \pi_{k+1} + \phi^* \delta \mid j \in J_{\delta}\}$.
  \item In particular, $ \alpha(\delta) = \sum_{j \in J_{\delta}} \alpha(-j \pi_{k+1} + \phi^* \delta)$
  and so for every $\delta \in \BFS(S)$ there is $j \in J_{\delta}$ such that $-j \pi_{k+1} + \phi^* \delta \in \BFS(X)$.
        \end{enumerate} 
    \end{enumerate}
\end{proposition}
\begin{proof}
   Recall that (a) follows from \autoref{prop.NefnessSmallContractions}. Only then do (b) and (c) have to be proved. Let us start with (b). Recall that $\phi$ is inert if and only if $b_{k+1}=1$ (\autoref{prop.CharacterizationInert}). Note that $P_i \cdot R \in \bZ$ for all $i = 1,\ldots,r$ and so $E \cdot R \in \bZ$ for all $\llbracket E \rrbracket \in N^1(X)$. Thus, if $\phi$ is inert, for any $\llbracket E \rrbracket = c_1 \pi_1 + \cdots + c_r \pi_r \in \FSupp(X)$, we have
     \[
     \bZ \ni E \cdot R = b_1 c_1+\cdots +b_kc_k-c_{k+1} > -1
     \]
     and so $E\cdot R \geq 0$. That is, $R \in \FE(X)$ if $\phi$ is inert. To prove the converse, for each $i=1,\ldots,k$, write
      \[
      \pi_{i} = -b'_{i} \pi_{k+1} + \phi^* \delta_i, \qquad b_i'\coloneqq b_i/b_{k+1}.
      \]
Then, for an element $\varepsilon = c_1 \pi_1+\cdots +c_r \pi_r \in \mathcal{Q}_X$, we have
\[
\varepsilon = (c_{k+1}-c_1 b'_1- \cdots -c_k b'_k) \pi_{k+1} + \phi^*( c_{1} \delta_{1} + \cdots + c_{k} \delta_k) + c_{k+2} \pi_{k+2} + \cdots +
c_{r} \pi_{r}.
\]
 It then follows that
   \[
   \mathcal{Q}_X \subset (-b,1) \pi_{k+1} \oplus \phi^*\langle \delta_1,\ldots,\delta_k, \pi_{k+2},\ldots,\pi_{r} \rangle_{[0,1)} \eqqcolon \mathcal{Q}'
   \]
   where $b\coloneqq b'_1+\cdots +b'_k$ and, with a slight abuse of notation, we write $\phi_! \pi_{i} = \pi_i$ for all $i=k+2,\ldots,r$. 
   
   Suppose now that $R \in \FE(X)$. By \autoref{prop.NonSmoothBlowups}, we may let $\varepsilon_1,\ldots,\varepsilon_k \in \FSupp(X)$ such that $\varepsilon_i \cdot R /b_{k+1} = b'_i-1 \geq 0$, in particular $b'_1,\ldots,b'_k \geq 1$. Since $k \geq 2$, this implies that a lattice point of $\mathcal{Q}'$ must intersect $R/b_{k+1}$ with integral values in the interval $[0,-b)$.
   However, $\pi_i$ is such a lattice point for every $i=1,\ldots,k$. This forces $b'_i$ to be integers for all $i=1,\ldots,k$, i.e., $b_{k+1} = 1$. This proves (b).

To prove (c), note that
   \[
   \mathcal{Q}_S=\langle \delta_1,\ldots,\delta_k, \pi_{k+2},\ldots,\pi_{r} \rangle_{[0,1)}.
   \]
   When $\phi$ is inert (i.e., $b'_i = b_i$), the above shows that $\varepsilon = c \pi_{k+1} + \phi^*\delta$ is a lattice point of $\mathcal{Q}_X$ only if $-c\in \{0,\ldots,b-1\}$ and $\delta$ is a lattice point of $\mathcal{Q}_S$. In fact, if we let $\delta = c_1\delta_1+ \cdots + c_k \delta_k + c_{k+2} \pi_{k+2} +\cdots + c_{r} \pi_{r}$, then
   \[
   -c=\lfloor c_1 b_1 + \cdots +c_k b_k \rfloor
   \]
   and $\varepsilon = c_1 \pi_1 + \cdots + c_r \pi_r$ with $c_{k+1}$ being the fractional part of $c_1b_1 + \cdots + c_k b_k$. The point here is that if $D$ is the $\bT$-invariant divisor $ c_1 P_1 + \cdots + c_k P_k + c_{k+2} P_{k+2} + \cdots + c_r P_r$ on $S$, then its strict transform $\tilde{D}$ (which is given by the same expression but considering the prime divisors on $X$) intersects $R$ with value $c_1 b_1 + \cdots + c_k b_k$. This means that $-c \in A_{\delta}$. In other words, if $\delta \in \FSupp(S)$ then its fiber in $\FSupp(X)$ is contained in $\{-j \pi_{k+1} + \phi^* \delta \mid j \in J_{\delta}\}$.

   For the remaining, converse inclusion, note that the above shows that the $j$ such that $-j \pi_{k+1} + \phi^* \delta \in \FSupp(X)$ are exactly those for which there is a $\bT$-invariant $[0,1)$-divisor $D$ on $S$ such that $\llbracket D \rrbracket = \delta$ and $\lfloor \tilde{D} \cdot R \rfloor = -j$. In particular, this holds for the extremal values $j=i_{\delta}$ and $j=k_{\delta}$. For the intermediate values, notice that the line segment from $-i_{\delta} \pi_{k+1} + \phi^* \delta$ to $-k_{\delta} \pi_{k+1} + \phi^* \delta$ is contained in $\mathcal{Q}_X$ as it is convex.
\end{proof}

Let us now illustrate how this proposition can be used to compute the Frobenius support of a weighted toric blowup of a projective space.

\begin{proposition}\label{prop.FrobSuppInertDivCont}
With notation as part (d) in \autoref{MainProposition}, suppose that $S=\bP^d$ and relabel the $\bT$-invariant prime divisors such that the extremal primitive relation is
\[R\: b_1u_1 + \dots + b_ku_k - u_{d+2} = 0,\] 
with $b_1 \leq b_2 \leq \dots \leq b_k$. Set $\FSupp(\bP^d)=\{l\eta \mid l=1,\ldots,d\}$ with $\eta$ the hyperplane class. Then, $i_l  \coloneqq i_{l\eta} = b_1 + \cdots + b_n$
where $n\coloneqq l - (d+1-k)$, so that $i_l = 0$ if and only if $l \geq d+1 -k$. Likewise, $k_l  \coloneqq k_{l\eta} = b_{k+1-l}+ \cdots + b_k -1$, so that $k_l = b_1 + \cdots + b_k-1$ if and only if $l\geq k$. In particular, for every $j\in \{0, \ldots, b_1+ \cdots +b_k-1\}$ there is $l\in\{1,\ldots, d\}$ such that $-j\pi_{d+2}+l\phi^* \pi_1$ is a big and nef element in $\FSupp(X)$. For $j=0$, we may take $l=1$.
\end{proposition}
\begin{proof}
For $l \in \{1,\ldots,d\}$, we have
\[
J_l = \{\lfloor c_1 b_1 + \cdots + c_k b_k \rfloor \mid c_1 + \cdots +c_{d+1} = l, 0 \leq c_1,\ldots, c_{d+1} <1 \}
\]
and $i_l = \min J_l$, $k_l = \max J_l$. Observe that for $\bm{c} = (c_1,\dots,c_{d+1}) \in [0,1)^{\times(d+1)}$ such that $|\bm{c}| = c_1 + \cdots + c_{d+1} = l$, we have
\[
c_1 + \cdots + c_k = l -(c_{k+1}+\cdots + c_{d+1}) > l-(d+1-k).
\]
\begin{claim}
    $\lfloor c_1 b_1 + \cdots + c_k b_k \rfloor \geq b_1 + 
\cdots +b_n$
\end{claim}
\begin{proof}[Proof of claim]
  It suffices to show that $c_1 + \cdots + c_k >n$ implies that $c_1 a_1 + \cdots + c_k a_k > a_1 + \cdots + a_n$ for every increasing sequence of integers $1\leq a_1 \leq \cdots \leq a_k$. We prove this by induction on $k\geq 2$. If $k=2$, then we have $2 > c_1 + c_2 > n$ and so $n \leq 1$. The case $n\leq 0$ is trivial, thus we may assume that $n=1$. In that case $c_1 > 1 - c_2$ and then
    \[
    c_1 a_1 + c_2 a_2 > (1-c_2) a_1 + c_2 a_2 = a_1 + (a_2-a_1) c_2 \geq a_1.
    \]
    This shows the base case $k=2$. For the inductive step, note that $c_1 + \cdots + c_k > n$ implies that $c_2 + \cdots + c_{k} > n - c_{1} > n-1$. The inductive hypothesis yields
    \begin{align*}
         c_1 a_1 + \cdots + c_{k} a_k >{} &(n-c_2-\cdots -c_k)a_1 +c_2 a_2 + \cdots  + c_k a_k\\
         ={} &n a_1 + (a_2-a_1)c_2 + \cdots + (a_k-a_1)c_k \\
         >{} &n a_1 + (a_2- a_1) + \cdots + (a_n-a_1) \\
         ={} &a_1 + a_2 + \cdots + a_n.
    \end{align*}
   The inductive hypothesis is applied for the last strict inequality. This proves the claim.
\end{proof}

Therefore, $i_{l} \geq b_1 + \cdots + b_n $. To prove the equality, define $\bm{c}$ as
\[
c_i \coloneqq \begin{cases}
    1-\epsilon &\text{for } i=1,\ldots,n,\\
    0 &\text{for } i=n+1,\ldots,k-1, \\
    \epsilon l &\text{for } i = k, \\
    1-\epsilon &\text {for } i = k+1,\ldots, d+1.
\end{cases}
\]
for some $0< \epsilon \ll 1/d$ still to be determined how small. Observe that
\[
|\bm{c}| = (1-\epsilon)n + \epsilon l + (1-\epsilon)(d+1-k) = l + (1-\epsilon)(n-l+d+1-k) = l .
\]
On the other hand,
\[
b_1 c_1 + \cdots + b_k c_k = (1-\epsilon)(b_1 + \cdots + b_n) + \epsilon l b_k = b_1 + \cdots + b_n + \epsilon (lb_k - b_1 - \cdots -b_n),
\]
where $m \coloneqq lb_k - b_1 - \cdots -b_n \in \bN$. Hence, taking $\epsilon < 1/m$ implies that $\lfloor b_1 c_1 + \cdots + b_k c_k \rfloor = b_1 + \cdots + b_n$ and so $i_{l} = b_1 + \cdots + b_n$.

The computation for $k_l$ is worked out analogously. Let $\bm{c}$ be as above with $|\bm{c}| = l$.

\begin{claim}
   $ \lfloor c_1 b_1 + \cdots + c_k b_k \rfloor \leq b_{k+1-l} + \cdots + b_k -1 $
\end{claim}
\begin{proof}
    It suffices to show that $c_1 + \cdots +c_{d+1} \leq l$ implies that $c_1 a_1 + \cdots + c_k a_k < a_{k+1-l} + \cdots + a_k$ for all increasing sequences of integers $1 \leq a_1 \leq \cdots \leq a_k$. This can be argued by induction on $k \geq 2$. The statement is trivial unless $l \leq k-1$. Suppose $k=2$, so $l=1$ and
    \[
    c_1 a_1 + c_2 a_2 \leq c_1 a_1 + (1-c_1)a_2 = a_2-(a_2-a_1)c_1<a_2.
    \]
    For the inductive step, note that $c_2 + \cdots + c_{d+1} \leq l$
    \begin{align*}
c_1 a_1 + \cdots + c_k a_k ={} & (l-c_2-\cdots-c_k) a_1 + c_2 a_2 + \cdots + c_{k} a_{k}   \\
={} & l a_1 + (a_2-a_1)c_2 + \cdots + (a_{k}-a_1)c_{k} \\
<{} & l a_1 + (a_{k+1-l}-a_1) + \cdots + (a_{k}-a_1) \\
={} & a_{k+1-l} + \cdots + a_{k},
\end{align*}
where the inductive hypothesis is applied in the last inequality. This proves the claim.
\end{proof}

This implies $k_l \leq b_{k+1-l} + \cdots + b_k -1$. To see that equality holds, we consider two cases. If $l \geq k$, take $\bm{c}$ such that $c_i = 1-\epsilon$ for all $i=1,\ldots,l$, $c_{l+1} =  \epsilon l$, and $c_i=0$ if $i\geq l+2$. If $l\leq k-1$, take $\bm{c}$ such that $c_1=\epsilon l$, $c_i = 1 -\epsilon$ for $i=k+1-l,\ldots,k$, and $c_i=0$ otherwise. In either case, we see that $|\bm{c}| = l$ and that by choosing $0<\epsilon \ll 1$ we get $\lfloor c_1 b_1 + \cdots +c_k b_k  \rfloor = b_{k+1-l} + \cdots + b_k -1$.
\end{proof}

\begin{example}\label{example.FatalExample}
We can use \autoref{prop.FrobSuppInertDivCont} to construct examples of singular $\bQ$-factorial toric varieties of Picard rank $2$ with big and nef Frobenius support. To provide a concrete example, consider the case in which $d=3$ and the relation is given by $R\: 3u_2 + 2u_3 - u_5 = 0$. Using \autoref{prop.FrobSuppInertDivCont}, we easily see that $i_1 = i_2 = 0$, $i_3 = 2$ while $k_1 = 2$, $k_2 = k_3 = 4$. The whole situation is depicted in \autoref{fig:NeronSeveriOfWeightedBlowUp}, from which we conclude that $\FSupp(X)$ is big and nef.
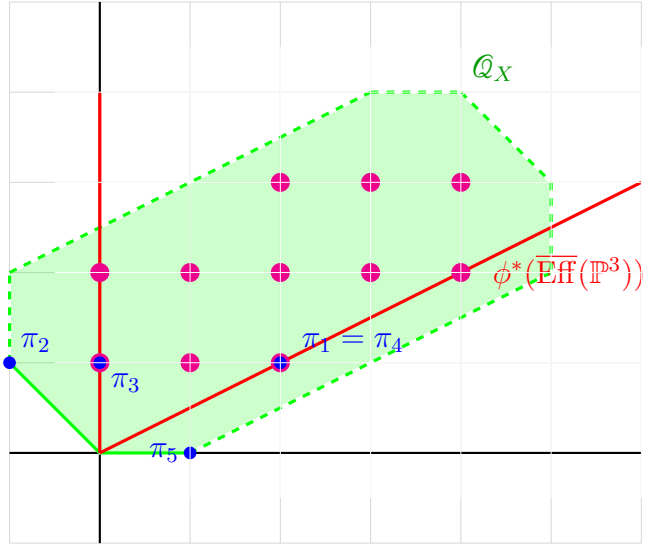
\begin{figure}[h]
    \centering
    \begin{tikzpicture}[scale=1.2, xscale = -1, rotate = 90]
\draw[step=1cm, gray!30, thin] (-1,-1) grid (5,6);

\draw[thick, black] (-1,0) -- (5,0); 
\draw[thick, black] (0,-1) -- (0,6); 

\draw[very thick, white, fill = green, fill opacity = 0.2] (0,0) -- (1,-1) -- (2,-1) -- (4,3) -- (4,4) -- (3,5) -- (2,5) -- (0,1) -- cycle;
\draw [very thick, green] (0,0) -- (1,-1);
\draw [very thick, green, dashed] (1,-1) -- (2,-1);
\draw [very thick, green, dashed] (2,-1) -- (4,3);
\draw [very thick, green, dashed] (4,3) -- (4,4);
\draw [very thick, green, dashed] (4,4) -- (3,5);
\draw [very thick, green, dashed] (3,5) -- (2,5);
\draw [very thick, green, dashed] (2,5) -- (0,1);
\draw [very thick, green] (0,1) -- (0,0);
Btw \draw[white] (4,4) circle (1pt) node[above right, text = green!60!black]{$\mathcal{Q}_X$};
\draw[very thick, red] (0,0) -- (3,6);
\draw [white, text = red] (2,5.2) node{$\phi^*(\eff(\P^3))$};
\draw[very thick, red] (0,0) -- (4,0);

\foreach \x/\y in {
  1/0, 1/1, 1/2, 
  2/0, 2/1, 2/2, 2/3, 2/4,
  3/2, 3/3, 3/4
}
  \fill[magenta] (\x,\y) circle (3pt);

\fill[blue] (1,0) circle (2pt) node[below right, blue] {$\pi_3$};
\fill[blue] (0,1) circle (2pt) node[left, blue] {$\pi_5$};
\fill[blue] (1,2) circle (2pt) node[above right, xshift = 4pt, blue] {$\pi_1 = \pi_4$};
\fill[blue] (1,-1) circle (2pt) node[above right, blue] {$\pi_2$};

\foreach \x in {1,2,3,4}
  \draw[white, very thin] (\x,-0.5) -- (\x,6.5);
\foreach \y in {1,2,3,4,5,6}
  \draw[white, very thin] (-0.5,\y) -- (4.5,\y);
\end{tikzpicture}
\caption{The Néron-Severi space of $X$ in \autoref{example.FatalExample} with its nef and pseudo-effective cones as well as Frobenius support. The pseudo-effective cone is generated by $\pi_2$ and $\pi_5$. The red rays generate the nef cone. In particular, $\pi_1$ generates the facet of $\nef(X)$ given by $\phi^*(\eff(\P^3))$. The green half-open polytope is $\mathcal{Q}_X$ and the purple lattice points are the elements of $\FSupp(X)$.}
\label{fig:NeronSeveriOfWeightedBlowUp}
\end{figure}
\end{example}

\begin{corollary} \label{cor.FSIntersectInternalFacetMovingCone}
    Let $X$ be a toric variety admitting a divisorial contraction $\phi = \phi_R\: X \rightarrow S$. Then, there exists $\llbracket E \rrbracket \in \BFS(X)$ such $E \cdot R = 0$.
\end{corollary}
\begin{proof}
We explain first why we may assume that $S = \bP^d$. The point is that the statement is local around the generic point of the center, and so we may replace $S$ by any other variety that shares the same affine chart around that point. This idea is borrowed from \cite{CarvajalRojasPatakfalviVarietieswithAmpleFrobenius-traceKernel}.

Note that the statement can be reinterpreted as the inequality 
   \[
   \sum_{\delta \in N^1(S)_{\bR}} \alpha(\phi^* \delta) = \sum_{\delta \in \FSupp(S)} \alpha(\phi^* \delta) > 0.
   \]
With notation as in \autoref{prop.LocalStructureExtremalDivContraction}, let $\phi_U \: X_U \to U$ be the restriction of $\phi$ to a purely $\bQ$-factorial toric affine chart $U \subset S$ that contains the generic point of $C$. Then, the restriction map $X_U \to X$ induces a homomorphism $N^1(X)_{\bR} \to N^1(X_U)_{\bR}$ that can be identified with the projection $N^1(X)_{\bR} \to \langle \pi_{k+1} \rangle_{\bR}$ from \autoref{eqn.NeronSeveriDivisorialContractions}. In particular,
\[
\sum_{\delta \in N^1(S)_{\bR}} \alpha(\phi^* \delta) = \sum_{\varepsilon \in N^1(X)_{\bR} \: \varepsilon|_{X_U} = 0} \alpha(\varepsilon)
\]
On the other hand, $\sE_{X,e}|_{X_U} = \sE_{X_U,e}$, hence this is further equal to $\alpha_{X_U}(0) \coloneqq \lim_{e \rightarrow \infty} t_e/q^d $ where $t_e$ is the rank the direct summand of $\sE_{X_U,e} = \sO(E_1) \oplus \cdots \oplus \sO(E_q)$ made up of those $\sO(E_i)$ such that $E_i\equiv 0$. 

In conclusion, the sum $\sum \alpha(\phi^* \delta)$ only depends on $U$, or more precisely on $\phi_U \: X_U \to U$. By \autoref{rem.FrobeniusSupportAndQuotients}, we see that it depends only on $\beta \: \Bl_H^{\mathrm{w}} \bA^d \to \bA^d$ as $\phi_U = \beta/G$; see 
\autoref{prop.LocalStructureExtremalDivContraction}. Therefore, $\sum \alpha(\phi^* \delta) = \sum \alpha(\psi^* \delta)$ where $\psi \: \Bl_H^{\mathrm{w}} \bP^d \to \bP^d$. In other words, we may assume that $S=\bP^d$.

To conclude, we explain the case $S = \bP^d$. Note that if $ \phi_R$ is an inert divisorial contraction, this is a direct consequence of \autoref{prop.FrobSuppInertDivCont}, indeed, $\phi^* \eta \in \BFS(X)$ is such an element. The non-inert is obtained in the same way. Indeed, let $u_1,\dots, u_{d+1}$ be the primitive ray generators of $\P^d$ which, by abuse of notation, we also consider as the primitive ray generators of the fan of $X$ and let $u_{d+2}$ be the remaining primitive ray generator. Then $R$ is given by the relation $R\: b_1u_1+ \dots + b_ku_k - b_{d+2}u_{d+2} = 0$ with $k < d$. Consider $D \coloneqq (1-\epsilon)P_{\P^d,d+1} + \epsilon P_{\P^d,1}$ with $1 \gg \epsilon > 0$, then $\llbracket D \rrbracket = \pi_1 = \eta \in \AFS(\P^d)$ is the hyperplane class. Then, 
   \[
   \phi^* \eta = (1-\epsilon)\pi_{X,d+1} + \epsilon \pi_{X,1} + \epsilon(b_{1}/b_{d+2})\pi_{X,d+2}.
   \] 
   Taking $\epsilon$ small enough such that $ \epsilon(b_{1}/b_{d+2}) < 1$, we conclude that $\phi^* \eta \in \FSupp(X)$. Since $\eta$ is ample, $\phi^* \eta$ is big and nef.
\end{proof}

\begin{remark}
   Suppose that in \autoref{cor.FSIntersectInternalFacetMovingCone} the map $\phi$ is a smooth blowup. From \cite[Proposition 4.2]{CarvajalRojasPatakfalviVarietieswithAmpleFrobenius-traceKernel}, it then follows that
\[
\sum_{[D] \in \FSupp(S)} m(\phi^* D;q) = q^c \binom{q+d-c}{d-c}
\]
where $c \coloneqq k-1$ is the codimension of $C$ in $S$. Therefore,
\[
\sum_{[D] \in \FSupp(S)} \alpha(f^*D) = \frac{1}{(d-c)!}>0.
\]
\end{remark}

\subsubsection{Mori Fibrations and $F$-effectiveness} \label{subsubsection.MorFibrationFeffectiveness}
 
Let $X$ be a $\bQ$-factorial toric variety and $\phi \: X \to S$ be a Mori fibration cut out by the relation $
R\: b_1u_1 + \cdots +b_k u_k=0$. In particular, $R \in  \FE(X)$. Relabel if necessary so that $b_1 \geq \cdots \geq b_k$. We have $\dim \langle u_1,\ldots, u_{k} \rangle_{\bR} = k-1$. Then, we may complete $u_1,\ldots,u_{k-1}$ to a basis $\mathcal{B} \coloneqq \{u_1,\ldots,u_{k-1},v_k,\ldots,v_d\} \subset N$ of $N_{\bR}$. This lets us write an exact sequence
\[
0 \to \langle u_1,\ldots,u_k \rangle_{\bR} \to N_{X,\bR} \to N_{S,\bR} \to 0,
\]
where $N_{-,\bR}$ denotes the $\bR$-linear spaces of one parameter subgroups of the variety $-$; the space where the respective fans live. Thus, we may think of $v_k,\ldots, v_d$ as a basis for $N_{S,\bR}$. Write the matrix of column vectors of $u_1,\ldots,u_r$ with respect to the basis $\mathcal{B}$ as
\[
[u_1 \cdots u_r] = \begin{bmatrix}
I_{k-1} & -\bm{a} & A & \\
\bm{0}_{(d-k+1)\times (k-1)} & \bm{0} & B & 
\end{bmatrix}
\]
where $\bm{a}$ is the column vector $[b_1/b_{k} \cdots b_{k-1}/b_{k}]^{\top}$. On the other hand, $A=(a_{ij})$ is a matrix of size $(k-1) \times (r-k)$. Likewise, $B$ has size $(d-k+1) \times (r-k)$ and its columns are the column vectors of the primitive ray generators of $S$ with respect to the basis $v_k,\ldots,v_d$.

The pullback map $\phi^* \: N^1(S)_{\bR} \to N^1(X)_{\bR}$ then corresponds to the induced map
\[
\coker B^{\top} \to \coker [u_1 \cdots u_r]^{\top}
\]
where $\pi_{i,S} \mapsto \pi_{k+i,X}$ for all $i=1,\ldots,r-k$. In particular, $\phi^*$ restricts itself to a map $N^1(S) \to N^1(X)$. Moreover, we have the relations
\[
\pi_{i} = a_i \pi_k - \phi^* \delta_i
\]
where $a_i \coloneqq b_i/b_k \geq 1$ and $\delta_i \coloneqq \sum_{j=1}^{r-k+1} a_{ij} \pi_{j}$. We then arrive at the decomposition
\[
N^1(X)_{\bR} = \bR \cdot \xi \oplus \phi^* N^1(S)_{\bR}.
\] 
where $\xi= \pi_k$ plays the role of the tautological class of the smooth case, in which case $\phi$ is a projective bundle.

Now, let $\varepsilon = c_1 \pi_1 +\cdots + c_r \pi_r \in \mathcal{Q}_X$ with $c_1,\ldots,c_r \in [0,1)$. Write the row vector $\bm{c} = [c_1 \cdots c_{k-1}]$. Then $-\bm{c} A = [c'_1 \cdots c'_{r-k}]$ is another row vector. Thus we have
\begin{align*}
    \varepsilon = &(\bm{c} \bm{a}+c_k) \xi + (c'_1+ c_{k+1})\pi_{k+1} + \cdots +
(c'_{r-k} + c_r)\pi_{r} \\
=&(\bm{c} \bm{a}+c_k) \xi + (c'_1+ c_{k+1}) \phi^*\pi_{1} + \cdots +
(c'_{r-k} + c_r)\phi^* \pi_{r-k}
\end{align*}
From this, it seems to be extremely difficult to relate $\mathcal{Q}_X$ with $\mathcal{Q}_S$ and so to relate $\FSupp(X)$ with $\FSupp(S)$. At the very least, we obtain the following improvement upon \autoref{scholium.BignessAndFibrations}.

\begin{proposition}[$F$-effectiveness of Mori fibrations] \label{prop.FEffectivenessMoriFibrations}
  With notation as above, the following two statements hold:
    \begin{enumerate}
        \item $\phi^* \FSupp(S) \subset  \FSupp(X)$.
        \item $\FSupp(X) \subset [0,b_1/b_k+\cdots +b_{k-1}/b_k+1) \xi \oplus \phi^* N^1(S)_{\bR}$.
    \end{enumerate}
    In particular, there is a point of $\FSupp(X)$ in the interior of the facet of $\eff(X)$ cut out by $R$.
\end{proposition}
\begin{proof}
    For the last assertion, take $\delta \in \BFS(X)$ so that $\phi^*\delta \in \FSupp(X)$ is the required point.
\end{proof}

\begin{corollary} \label{cor.DimensionOfFrobeniusCone}
    If $X$ is a $\bQ$-factorial toric variety, then $\dim \Frob(X) = \rho$. In other words,  $\FE(X)$ is strongly convex.
\end{corollary}
\begin{proof}
    Putting \autoref{cor.FSIntersectInternalFacetMovingCone} and \autoref{prop.FEffectivenessMoriFibrations} together, we see that $\FSupp(X)$ intersects the interior of every facet of the moving cone of divisors, which is a $\rho(X)$-dimensional cone. These elements of $\FSupp(X)$ then span a $\rho(X)$-dimensional space and so $\dim \Frob(X) = \rho(X)$.
\end{proof}

\begin{corollary} \label{cor.PullingBackIndependentClassesFibrations}
    With notation as in \autoref{prop.FEffectivenessMoriFibrations}, there are linearly independent elements $\varepsilon_1,\ldots,\varepsilon_{\rho -1} \in \FSupp(X)$ such that $\varepsilon_i \cdot R = 0$ for all $i=1,\ldots,n$. Therefore, $R$ is an extremal ray of $\FE(X)$.
\end{corollary}

\begin{remark}
    This is what we meant when we said that extremal divisorial contractions and Mori fibrations exhibit opposite behaviors. In the former case, we can compare $\FSupp(X)$ and $\FSupp(S)$ very well, but it is very difficult to pull back elements of $\FSupp(S)$ to elements in  $\FSupp(X)$. In the latter case, the exact opposite occurs.
\end{remark}

\subsubsection{Main Theorem}

Putting everything together, we obtain at once our main result.

\begin{theorem}[Main Theorem] \label{thm.MainTheorem}
    Let $X$ be a $\bQ$-factorial toric variety. The following holds:
\begin{enumerate}
    \item $\FSupp(X)$ moves if and only if all toric small $\Q$-factorial modifications of $X$ are divisorially inert. In that case, 
    \[
    \Frob(X) 
    \subset \mov^1(X)
    \]
    and they share every facet that is contained in $\partial \eff(X)$.
    \item $\FSupp(X)$ is nef if and only if $X$ is a birationally inert Fano variety. In that case, the cones $\Frob(X)
    \subset \nef(X)$ share every facet that is contained in $\partial \eff(X)$. Moreover, there is a finite sequence of inert extremal divisorial contractions
\[
X \to X_1 \to X_2 \to X_3 \to \cdots \to X_n
\]
such that $\nef(X_n) = \eff(X_n)$ and $n < \rho(X)$.
    \item  $\FSupp(X)$ is ample if and only if $\rho(X)=1$, i.e., $X$ is a prime Fano variety.
\end{enumerate}
\end{theorem}
\begin{proof}
    The first statement of (a) follows by putting together \autoref{cor.FrobeniusSupport} and \autoref{MainProposition}. Its second statement then follows from \autoref{cor.PullingBackIndependentClassesFibrations}. Part (b) follows from (a) and the fact that the nefness of the Frobenius support is inherited down through inert extremal divisorial contractions; see part (d) of \autoref{MainProposition}. Part (c) is a direct consequence of part (b) and \autoref{scholium.BignessAndFibrations}.
\end{proof}

\begin{remark}
    According to Fujino--Sato \cite[Proposition 5.3]{FujinoSatoToricVarietiesWhoseNefBundlesAreBig}, the condition $\nef(X) = \eff(X)$ in part (b) of \autoref{thm.MainTheorem} means that $X$ admits a finite toric cover from a product of projective spaces (such a cover being an isomorphism if $X$ is smooth). However, we may adapt their proof to conclude that $\nef(X) = \eff(X)$ also means that $X$ is a product of varieties of Picard rank $1$.
\end{remark}

\begin{remark}
    In \autoref{thm.MainTheorem}, everything seems to indicate that we actually obtain an equality between the cone of $F$-effective divisors and the one of moving divisors. For instance, we conjecture that \autoref{cor.PullingBackIndependentClassesFibrations} also holds for inert divisorial contractions. On the other hand, it also seems that $\nef(X) \subset \Frob(X)$ if $X$ is Fano. If this were true, then we could also say $X$ is a birationally inert Fano variety if and only if $\Frob(X) = \nef(X)$. In sorting this out, answering the following question might help.
\end{remark}

\begin{question}
    What is the intersection of $\FE(X)$ and $\moriCone(X)$? For instance, is there a way to describe $F$-effective $1$-cycles using certain primitive relations in analogy to Batyrev's description of the Mori cone?
\end{question}

\begin{remark}
It remains open to characterize when exactly $\sE_{X,e}$ is big (resp. big and nef) among $\bQ$-factorial toric varieties. From \autoref{thm.MainTheorem} and \autoref{prop.FEffectivenessMoriFibrations}, we conclude that if $\sE_{X,e}$ is big and nef then there is a finite sequence of inert divisorial contractions ending up in a prime Fano variety. But which sequences are allowed to preserve $\sE_{X,e}$ is big and nef? See \autoref{example.FatalExample}.
\end{remark}

\section{Ample $F$-signature and Homogeneity} \label{sec.AmpleFSignHomogeneity}

Aiming to characterize homogeneity among smooth toric varieties, we introduce the notion of \emph{ample $F$-signature} for
toric varieties. Although we believe this invariant can be defined for general varieties, we restrict ourselves to this case, where it is a much simpler task. The main difficulty lies in coming up with a suitable notion of \emph{ample rank} for locally free (or, more generally, reflexive) sheaves. However, whatever this notion might be, it should satisfy the following. To a locally free (or reflexive) sheaf $\sE$ we should attach a non-negative integer $\ark \sE$ such that the following three properties hold:
\begin{enumerate}
    \item $\ark \sE \leq \rk \sE$,
    \item $\ark \sE =\rk \sE$ if and only if $\sE$ is ample, and
    \item $\ark (\sE' \oplus \sE'') = \ark \sE' \oplus \ark \sE''$,
\end{enumerate}
In particular, if $\sE = \bigoplus_{i=1}^n \sL_i$ splits as a direct sum of invertible (or reflexive of rank $1$) sheaves $\sL_i=\sO_X(D_i)$, then
\[
\ark \sE = |\{i \in \{1,\ldots,n\} \mid D_i \text{ is ample}\}|.
\]

Since we are only dealing with split sheaves as such, we consider it instructive to define ample rank only for them. Doing otherwise would be an unnecessarily lengthy tangent for now.

This lets us define the following sequence for any variety $X$ such that $\sE_{X,e}$ is fully split:
\[
a_e(X) \coloneqq \ark \sE_{X,e} \leq q^d-1, \quad 0 \neq e \in \bN. 
\]
This includes toric varieties, but also homogeneous spaces such as ordinary abelian varieties \cite{SannaiTanakaOrdinaryAbelianFrobenius,EjiriSannaACharacterizationOrdnaryAblianVarieties}. For a $d$-dimensional toric variety $X$, we have that
\[
a_e(X) = \mathcal{a}(X) q^d + O(q^{d-1})
\]
where 
\[
\mathcal{a}(X) \coloneqq \sum_{\llbracket E \rrbracket \in  \AFS(X)} \alpha(E)
\]
and $\alpha(E)$ are as in \autoref{cor.multiplicityOfDirectSummands}. In particular, we may define the \emph{ample $F$-signature} of $X$ to be $\mathcal{a}(X) \in [0,1]\cap \bQ$. 

\begin{example}[Ample $F$-signature of Hirzebruch surfaces]
    Let $S_n$ be the projective bundle over $\bP^1$ given by $\bP(\sO\oplus \sO(-n))$. From \cite[\S4.2]{CarvajalRojasPatakfalviVarietieswithAmpleFrobenius-traceKernel}, it immediately follows that $\mathcal{a}(S_n) = 1/n$.
\end{example}

We readily obtain the following.
\begin{proposition} \label{prop.ExtremalValuesAmpleSignature}
    Let $X$ be a $\bQ$-factorial toric variety. Then,
\begin{enumerate}
        \item $\mathcal{a}(X)=0$ if and only if $\AFS(X) = \emptyset$
        \item  $\mathcal{a}(X)=1$ if and only if $\AFS(X) = \BFS(X)$.
    \end{enumerate}
\end{proposition}

Our next task is to give more geometric meaning to these extremal values for the ample $F$-signature. For its positivity, we use \autoref{cor.ExistenceOfAmpleClassesInFrobSupport} to obtain the following result directly.

\begin{corollary}[Positivity of the ample $F$-signature] \label{cor.PositivityAmpleFSignature}
    Let $X$ be a $\bQ$-factorial toric variety. Then, its ample $F$-signature $\mathcal{a}(X)$ is positive if and only if there is a toric log Fano pair $(X,\Delta)$ of class index $1$.
\end{corollary}

We now focus on the maximality of the ample $F$-signature and its relationship with homogeneity. One direction is clear: homogeneous spaces have ample $F$-signature equal to $1$. In particular, ample $F$-signature equal to $1$ does not imply that $\sE_e$ is ample for all $e>0$. This is in contrast to what happens in the local case with the $F$-signature, where a local ring $R$ having $F$-signature equal to $1$ is equivalent to $F^e_* R $ being free for all $e>0$ \cite[Corollary 16]{HunekeLeuschkeTwoTheoremsAboutMaximal}. Thus, one may wonder why the local argument for the $F$-signature does not work in our global, projective setup. Let us try to adapt it to our case and see what it yields. In what follows, we use the notation introduced in \autoref{rem.AsymptoticBehavior} and assume that $X$ is smooth. 

Let us start by writing
\[
\sE_{X,e+e'} =  \sE_{X,e} \oplus \sF_{e,e'} = \sE_{X,e} \oplus \bigoplus_{[E] \in \FSupp(X)} \sE_{X,e}[E]^{\oplus m(E;q')},
\]
where $\sE_{X,e}[E]\coloneqq F^e_*\sO_X(E+(1-q)K_X)$. Then
\begin{equation} \label{eqn.Decomposition}
\frac{\ark \sE_{X,e+e'}}{q^dq'^d} = \frac{\ark \sE_{X,e}}{q^dq'^ d} + \sum_{[E] \in \FSupp(X)} \frac{m(E;q')}{q'^d} \frac{ \ark \sE_{X,e}[E]}{q^d}.
\end{equation}
Taking the limit $e' \rightarrow \infty$ yields
\[
\mathcal{a}(X) = \sum_{[E] \in \BFS(X)} \alpha(E) \frac{ \ark \sE_{X,e}[E]}{q^d}.
\]
This implies that $\mathcal{a}(X)=1$ if and only if for some/all $e$ it follows that $\sE_{X,e}[E]$ is ample for all $[E] \in \BFS(X)$.

On the other hand, taking the limit $e \rightarrow \infty$ in \autoref{eqn.Decomposition} yields
\[
\mathcal{a}(X) = \frac{\mathcal{a}(X)}{q'^d} + \sum_{[E] \in \FSupp(X)} \frac{m(E;q')}{q'^d} \mathcal{a}_E(X)
\]
where 
\[
\mathcal{a}_E(X) \coloneqq \lim_{e \rightarrow \infty} \frac{ \ark \sE_{X,e}[E]}{q^d},
\]
assuming that these limits exist. They do in fact exist by an argument similar to that showing the existence of $\mathcal{a}(X)$. Alternatively, one may use $\limsup$ or $\liminf$ instead. Anyhow, the conlusion we want to drop from this is that $\mathcal{a}(X)=1$ if and only if $\mathcal{a}_E(X)=1$ for all $[E] \in \FSupp(X)$.\footnote{Where we are claiming that $\mathcal{a}_E(X)$ exist and it is equal to $1$.} Let us summarize everything as follows.

\begin{proposition}\label{prop.NaiveInterpretationOfMaximalAmpleFsIGNATURE}
    Let $X$ be a smooth toric variety. The following statements are equivalent:
    \begin{enumerate}
        \item $\mathcal{a}(X)=1$.
        \item There is $e>0$ such that $\sE_{X,e}[E]$ is ample for all $[E] \in \BFS(X)$.
        \item $\sE_{X,e}[E]$ is ample for all $e>0$ and all $[E] \in \BFS(X)$.
        \item $\mathcal{a}_E(X)=1$ for all $E \in \FSupp(X)$.
    \end{enumerate}
\end{proposition}

\begin{remark}
Taking the limit $e,e' \rightarrow \infty$ in \autoref{eqn.Decomposition} yields
\[
\mathcal{a}(X) =  \sum_{[E] \in \BFS(X)} \alpha(E) \mathcal{a}_E(X).
\]    
\end{remark}

\begin{example} \label{example.Sharpness} The following shows that the quantifier on $E$ in (c) of \autoref{prop.NaiveInterpretationOfMaximalAmpleFsIGNATURE} cannot be changed to \emph{for all $[E] \in \FSupp(X)$}, which means that $\sE_{X,e
}$ is ample. Let $X = \bP^1 \times \bP^ 1$. Then \[
\sE_{X,e} = \sO(1,0)^{\oplus(q-1)}\oplus \sO(0,1)^{\oplus(q-1)} \oplus \sO(1,1)^{\oplus(q-1)^2}.\]
In particular, $\FSupp(X)=\{(1,0),(0,1),(1,1)\}$. Moreover,
\[
\sE_{X,e}(1,0) = \sO(1,0)^{\oplus q } \oplus \sO(1,1)^{\oplus q(q-1)}, \quad  \sE_{X,e}(1,0) = \sO(0,1)^{\oplus q } \oplus \sO(1,1)^{\oplus q(q-1)},
\]
and
\[
\sE_{X,e}(1,1) = \sO(1,1)^{\oplus(q-1)^2}.
\]
So $\sE_{X,e}(1,1)$ is ample while $\sE_{X,e}(1,0)$ and $\sE_{X,e}(0,1)$ are not, which shows that \autoref{prop.NaiveInterpretationOfMaximalAmpleFsIGNATURE} is sharp. However, $\sE_{X,e}(1,0)$ and $\sE_{X,e}(0,1)$ are nef.
\end{example}

Our closing theorem is the following.

\begin{theorem} [Maximality of the ample $F$-signature]\label{thm.CharacterizationHomogenousSpaces}
    Let $X$ be a $\bQ$-factorial toric variety. Then, $\mathcal{a}(X)=1$ if and only if $\eff(X) = \nef(X)$. In particular, for $X$ smooth, $\mathcal{a}(X)=1$ if and only $X$ is homogeneous. 
\end{theorem}
\begin{proof}
    If $\eff(X) = \nef(X)$ then big divisors are ample and so $\mathcal{a}(X) = 1$. Conversely, suppose that $\mathcal{a}(X)=1$, that is, $\BFS(X) \subset \ample(X)$. As a direct application of \autoref{cor.FSIntersectInternalFacetMovingCone}, we conclude that the facets of the cone of moving divisors correspond to Mori fibrations (on the toric small $\bQ$-factorial modifications of $X$). This readily implies that the cone of moving divisors coincides with $\eff(X)$---consider the statement on the dual cones. In other words, we see that all effective divisors move. 

    In this way, all it remains to prove is that $\FSupp(X)$ is nef. To do so, it suffices to prove that $E \cdot R \geq 0$ for all $\varepsilon \coloneqq \llbracket E \rrbracket \in \FSupp(X) \cap \partial \eff(X)$ and all non-effective extremal primitive relations $R$. This is achieved in a couple of steps. Let $F$ be the facet of $\eff(X)$ where $\varepsilon$ sits.

    \begin{claim}
    There is $j=1,\ldots,r$ such that $\pi \coloneqq \pi_j$ satisfies $\pi \notin F$ and $\pi \cdot R \leq 0$.
    \end{claim}
    \begin{proof}[Proof of claim]
        Suppose, for the sake of contraction, that this is not true. That is, for all $\pi_j \neq F$ we have $\pi_j \cdot R > 0$. That is to say that if $R$ is of the form 
        \[
        R\: b_1u_1 + \dots + b_ku_k - b_{k+1}u_{k+1} - \dots - b_lu_l = 0
        \]
        (with $l\geq k+1$ by assumption) then $F$ must be cut out by an effective relation among the $u_1,
        \ldots,u_k$, say
        \[
        R' \: b'_1 u_1 + \cdots + b'_{k'} u_{k'} = 0.
        \] 
        
        If $k=k'$ then the primitive relation defined by the primitive collection ${u_1,
         \ldots,u_k}$ is $R'$, violating $l \geq k+1$. Hence, $k'<k$ in which case $\sigma \coloneqq \langle u_{s_1}, \dots, u_{s_m}\rangle_{\R \geq 0}$ is a cone in $\Sigma_X$ as $\{u_1,\dots, u_k\}$ is a primitive collection. However, the relation $R'$ would then contradict the strong convexity of $\sigma$. This proves the claim.
    \end{proof}

\begin{claim}
    $\varepsilon + \pi \in \BFS(X)$
\end{claim}
\begin{proof}[Proof of claim]
    Since $\pi$ moves, we can write is as $\pi = \sum_{i \neq j} a_i \pi_i$ for some $a_i \in \bR_{\geq 0}$; see \autoref{prop.MovingConeToricVarieties}. Write $\varepsilon$ as $\sum_{i=1}^r c_i \pi_i$ for some $c_1,\ldots,c_r \in [0,1)$. Then note that
    \[
    \varepsilon + \pi = \varepsilon + (1-1/n)\pi + \epsilon\pi = \sum_{i\neq j}(c_i +  a_i/n)\pi_i + (1-\epsilon)\pi_j.
\]
Taking $n \gg 0$ such that $c_i + a_i/n < 1$ for all $i\neq j$, we conclude that $\varepsilon+ \pi \in \FSupp(X)$. Moreover, $\varepsilon+ \pi$ is big as otherwise $\pi \in F$; which contradicts the construction of $\pi$.
\end{proof}

   With the above two claims, we are ready to conclude as follows:
    \[
    E \cdot R = (\varepsilon + \pi)\cdot R - \pi \cdot R > 0
\]
    where the inequality holds as $\varepsilon + \pi$ is then ample and $\pi \cdot R \leq 0$ as $j \notin \{1,\dots, k\}$. 
\end{proof}

\subsection{Nef $F$-signature and further problems}
In the previous discussion, there was nothing special about ampleness over nefness. We could have written nef instead of ample, \emph{nef rank} instead of ample rank, $\mathcal{n}$ instead of $\mathcal{a}$, etc., and the same arguments are valid. In particular, let us define the \emph{nef rank $\nrk \sE$} of a split locally free sheaf $\sE$ analogously to the ample rank---counting nef invertible summands. Then, we define the \emph{nef $F$-signature} of $X$ as
\[
\mathcal{n} (X)\coloneqq \lim_{e \rightarrow \infty} \frac{\nrk \sE_{X,e}}{q^d} = \sum_{[E] \in  \NFS(X)} \alpha(E) \in [0,1] \cap \bQ.
\]
Notice that the sum traverses all the big and nef divisors in the Frobenius support of $X$. 

Then \autoref{prop.NaiveInterpretationOfMaximalAmpleFsIGNATURE} holds \emph{verbatim} by replacing ``ample'' with ``nef'' and ``$\mathcal{a}$'' with ``$\mathcal{n}$.'' Similarly, in analogy to \autoref{prop.ExtremalValuesAmpleSignature}, we have the following.
\begin{proposition} \label{prop.ExtremalValuesAmpleSignature}
    Let $X$ be a $\bQ$-factorial toric variety. Then,
\begin{enumerate}
        \item $\mathcal{n}(X)=0$ if and only if $\NFS(X) \cap \BFS(X) = \emptyset$. That is, $\mathcal{n}(X)>0$ if and only if $X$ admits a KLT toric log pair $(X,\Delta)$ of index $1$ such that $-(K_X+\Delta)$ is big and nef (aka weakly log Fano pair).
        \item  $\mathcal{n}(X)=1$ if and only if $\BFS(X) \subset \nef(X)$.
    \end{enumerate}
\end{proposition}

Of course, $\mathcal{n}(X) \geq \mathcal{a}(X)$. However, one disadvantage that $\mathcal{a}$ might have over $\mathcal{n}$ is that it vanishes way more often, in which case it is not as useful as a measurement tool. See, for example, what happens for del Pezzo surfaces. In fact, based on empirical evidence, it becomes much rarer for $\mathcal{a}(X)$ to be nonzero as $\dim X$ increases. However, it may happen that $\mathcal{n}(X)=0$ even for surfaces, as the following example shows. The authors thank Fabio~Bernasconi for kindly suggesting this example to us.

\begin{example}[Toric surface with zero nef signature]
    Let $Y$ be the smooth toric surface given by blowing up $\P^2$ along the three torus-invariant points. Blow up two torus-invariant points of $Y$ to obtain the toric variety $X$ such that the rays of $\Sigma_X$ are genereted by the vectors $(1,0),(1,1),(0,1),(-1,0),(-1,-1),(0,-1),(1,-1)$ and the maximal dimensional cones are the obvious ones. Using the computer algebra software Macaulay2 \cite{M2}, we obtain $\mathcal{n}(X) = 0$.  
\end{example}

However, all the examples we have looked at support affirmative answers for the following.

\begin{question}\label{conjecture.OnMaximalityOfNefFSignature}
    Let $X$ be a $\bQ$-factorial toric variety. Do the following statements hold?
\begin{enumerate}
    \item If $X$ is Fano then $\mathcal{n}(X)>0$.  
    \item  $\mathcal{n}(X)=1$ only if $\sE_{X,e}$ is nef for all $e>0$ (i.e., $X$ is a birationally inert Fano variety).
\end{enumerate}
\end{question}

We close this article with a problem regarding the behavior of ample $F$-signatures under inert divisorial contractions. With notation as (c) in \autoref{MainProposition}, observe that

\[
\mathcal{a}(S) = \sum_{ \delta \in \AFS(S)} \alpha(\delta)   = \sum_{ \substack{\delta \in \AFS(S) \\ j \in J_{\delta}}} \alpha(-j\pi_{k+1} + \phi^* \delta)   \geq  \sum_{ \substack{\delta \in \AFS(S) \\ 0 \neq j \in J_{\delta}}} \alpha(-j\pi_{k+1} + \phi^* \delta) = \mathcal{a}(X).
\]
Moreover, $\mathcal{a}(S) > \mathcal{a}(X)$ if and only if there is $\delta \in \AFS(S)$ with $i_{\delta} = 0$. All examples seem to indicate that this is the case as long as $\mathcal{a}(X)>0$. Recall that we know how to find $\delta \in \BFS(X)$ with $i_{\delta} = 0$; see \autoref{cor.FSIntersectInternalFacetMovingCone}. The challenge is to find such $\delta$ in $\AFS(X)$ whenever $\mathcal{a}(X)>0$.

\bibliographystyle{skalpha}
\bibliography{MainBib}

\end{document}